\documentclass[12pt]{article}

\usepackage{fontspec, xunicode, xltxtra}  
\usepackage{mathtools}
\usepackage{amsmath,diagbox}
\usepackage{indentfirst}
\usepackage{mathrsfs}
\usepackage{amsfonts}
\usepackage{arydshln}
\usepackage{enumerate}
\usepackage{setspace}
\usepackage{amssymb,amsthm,cases}
\usepackage{amsbsy,pifont}
\usepackage{latexsym,diagbox,tabularx}
\usepackage{amsmath,amscd,bbm,multirow}
\usepackage{graphicx,epsfig,extarrows,dsfont,mathtools}
\usepackage[final,allcolors=blue,colorlinks=true]{hyperref}
\usepackage[normalem]{ulem}
\input xy
\usepackage{caption,subcaption} 
\usepackage{tikz}
\usepackage{tikz-cd}
\usetikzlibrary{decorations.pathreplacing,decorations.markings}
\usetikzlibrary{cd}
\usetikzlibrary{shapes.geometric}
\usetikzlibrary{arrows,arrows.meta}
\usetikzlibrary{graphs,positioning}
\usetikzlibrary{matrix,decorations.pathmorphing}
\usetikzlibrary{math,calc,intersections,through,angles,arrows.meta,shapes.geometric,shadows,quotes,spy,datavisualization,datavisualization.formats.functions,plotmarks}
\tikzset{every picture/.style={samples=300,smooth,line join=round,thick,>=stealth}}
\usepackage{enumerate}
\usepackage{fancyhdr}
\usepackage[sectionbib]{chapterbib}

\tikzset{
	on each segment/.style={
		decorate,
		decoration={
			show path construction,
			moveto code={},
			lineto code={
				\path [#1]
				(\tikzinputsegmentfirst) -- (\tikzinputsegmentlast);
			},
			curveto code={
				\path [#1] (\tikzinputsegmentfirst)
				.. controls
				(\tikzinputsegmentsupporta) and (\tikzinputsegmentsupportb)
				..
				(\tikzinputsegmentlast);
			},
			closepath code={
				\path [#1]
				(\tikzinputsegmentfirst) -- (\tikzinputsegmentlast);
			},
		},
	},
	mid arrow/.style={postaction={decorate,decoration={
				markings,
				mark=at position .5 with {\arrow[#1]{stealth}}
	}}},
}

\numberwithin{equation}{section}

\theoremstyle{plain}
\newtheorem{theorem}{Theorem}[section]

\newtheorem{corollary}[theorem]{Corollary}

\newtheorem{lemma}[theorem]{Lemma}

\theoremstyle{definition}
\newtheorem{definition}[theorem]{Definition}

\newtheorem{remark}[theorem]{Remark}

\newtheorem{notation}[theorem]{Notation}

\def\XXint#1#2#3{{\setbox0=\hbox{$#1{#2#3}{\int}$}
		\vcenter{\hbox{$#2#3$}}\kern-.5\wd0}}

\DeclareMathSymbol{\subseteq}{\mathrel}{symbols}{"12}
\DeclareMathSymbol{\supseteq}{\mathrel}{symbols}{"13} 
\DeclareMathSymbol{\subsetneq}{\mathrel}{AMSb}{"28}                  \DeclareMathSymbol{\supsetneq}{\mathrel}{AMSb}{"29}    

\DeclareMathSymbol{\nsubseteq}{\mathrel}{AMSb}{"2A}                  

\DeclareMathSymbol{\nsupseteq}{\mathrel}{AMSb}{"2B}

\DeclareMathDelimiter{\langle}{\mathop}{symbols}{"68}{largesymbols}{"0A}
\DeclareMathDelimiter{\rangle}{\mathclose}{symbols}{"69}{largesymbols}{"0B}

\DeclareSymbolFont{txfontsA}{U}{txmia}{m}{it}
\SetSymbolFont{txfontsA}{bold}{U}{txmia}{bx}{it}
\DeclareFontSubstitution{U}{txmia}{m}{it}
\DeclareMathSymbol{\upalpha}{\mathord}{txfontsA}{"0B}
\DeclareMathSymbol{\upbeta}{\mathord}{txfontsA}{"0C}
\DeclareMathSymbol{\upgamma}{\mathord}{txfontsA}{"0D}
\DeclareMathSymbol{\updelta}{\mathord}{txfontsA}{"0E}
\DeclareMathSymbol{\upepsilon}{\mathord}{txfontsA}{"0F}
\DeclareMathSymbol{\upzeta}{\mathord}{txfontsA}{"10}
\DeclareMathSymbol{\upeta}{\mathord}{txfontsA}{"11}
\DeclareMathSymbol{\uptheta}{\mathord}{txfontsA}{"12}
\DeclareMathSymbol{\upiota}{\mathord}{txfontsA}{"13}
\DeclareMathSymbol{\upkappa}{\mathord}{txfontsA}{"14}
\DeclareMathSymbol{\uplambda}{\mathord}{txfontsA}{"15}
\DeclareMathSymbol{\upmu}{\mathord}{txfontsA}{"16}
\DeclareMathSymbol{\upnu}{\mathord}{txfontsA}{"17}
\DeclareMathSymbol{\upxi}{\mathord}{txfontsA}{"18}
\DeclareMathSymbol{\uppi}{\mathord}{txfontsA}{"19}
\DeclareMathSymbol{\uprho}{\mathord}{txfontsA}{"1A}
\DeclareMathSymbol{\upsigma}{\mathord}{txfontsA}{"1B}
\DeclareMathSymbol{\uptau}{\mathord}{txfontsA}{"1C}
\DeclareMathSymbol{\upupsilon}{\mathord}{txfontsA}{"1D}
\DeclareMathSymbol{\upphi}{\mathord}{txfontsA}{"1E}
\DeclareMathSymbol{\upchi}{\mathord}{txfontsA}{"1F}
\DeclareMathSymbol{\uppsi}{\mathord}{txfontsA}{"20}
\DeclareMathSymbol{\upomega}{\mathord}{txfontsA}{"21}
\DeclareMathSymbol{\upvarepsilon}{\mathord}{txfontsA}{"22}
\DeclareMathSymbol{\upvartheta}{\mathord}{txfontsA}{"23}
\DeclareMathSymbol{\upvarpi}{\mathord}{txfontsA}{"24}
\DeclareMathSymbol{\upvarrho}{\mathord}{txfontsA}{"25}
\DeclareMathSymbol{\upvarsigma}{\mathord}{txfontsA}{"26}
\DeclareMathSymbol{\upvarphi}{\mathord}{txfontsA}{"27}                   

\DeclareSymbolFont{ugmL}{OMX}{mdugm}{m}{n}
\SetSymbolFont{ugmL}{bold}{OMX}{mdugm}{b}{n}
\DeclareMathAccent{\wideparen}{\mathord}{ugmL}{"F3}

\def\pt{\partial}

\def\ra{\rightarrow}

\def\s{\subseteq}

\def\e{\epsilon}

\def\ol{\overline}

\def\bf{\textbf}
\def\pt{\partial}

\def\Om{\Omega}
\def\la{\lambda}
\def\al{\alpha}
\def\be{\beta}
\def\de{\delta}

\def\Ga{\Gamma}
\def\La{\Lambda}

\def\ts{\times}
\def\ka{\kappa}

\def\iy{\infty}

\def\ots{\otimes}
\def\f{\frac}

\def\df{\mathrm d}

\def\wh{\widehat}

\def\hra{\hookrightarrow}

\def\dra{\downarrow}
\def\mcL{\mathcal{L}}

\def\mcA{\mathcal{A}}

	\DeclareMathOperator{\Div}{div}

	\DeclareMathOperator{\dist}{dist}

	\DeclareMathOperator{\supp}{{supp}}

	\newcommand{\R}{\mathbb R}
	\newcommand{\N}{\mathbb N}

\begin{document} 
	
	
	\title{Shape Design for    Degenerate Hyperbolic  Equation with Degenerate  Boundary and Its Application to Observability}
	\author{Dong-Hui Yang, Yuanzhi Zhou\footnote{\small
			The corresponding author. Email: donghyang@outlook.com}\\
			{\it School of Mathematics and Statistics, Central South University}\\
		{\it Changsha 410075, P.R.China}\\}

	\maketitle{}
	\thispagestyle{empty}
	\thispagestyle{empty}
	
	\begin{abstract}
		 In this paper, we study the observability and controllability of a class of degenerate hyperbolic equations with a control region intersecting the degenerate set. Unlike the existing results that mainly deal with control regions separated from the degeneracy, we consider the case where the control region reaches the degenerate part. To handle the difficulty caused by the degeneracy, we introduce a shape-design-based approximation method  based on shape design by approximating the degenerate equation with a family of uniformly hyperbolic equations. The proof relies on the spectral approximation of the associated operators, precise estimates for weak solutions, and the multiplier method. We first establish observability inequalities for the approximating equations with constants independent of the approximation parameter. Then, by passing to the degenerate limit, we obtain the observability inequality for the original degenerate equation.

		 \vspace{0.3cm}
		 
		 
		 \vspace{0.3cm}
		 
	\end{abstract}
	
	\section{Introduction}

	Controllability for uniformly parabolic and hyperbolic equations, as well as for stochastic differential equations, has been extensively studied; see, for example, \cite{Komornik,Lasiecka,Lasiecka1,Lions,Lu,Weiss,Yang5,Yao,Zuazua}. A fundamental tool in the study of controllability for deterministic systems is the Hilbert Uniqueness Method (HUM), introduced by J.-L. Lions, which establishes a connection between exact controllability of a system and an observability inequality for its corresponding adjoint system. Consequently, the study of controllability can often be reduced to establishing suitable observability estimates, \cite{Lions,Yang1}.
	
	Controllability for degenerate parabolic and hyperbolic equations has also been extensively studied; see, for example, \cite{Alabau,Alabau1,Akil,Bai,Buffe,Cannarsa,Cannarsa1,Fragnelli,Gueye,Yang,Yang4,Zhang}. Most of these works concern one-dimensional problems, with the notable exception of the higher-dimensional case considered in \cite{Yang4}.
	
	The controllability of degenerate hyperbolic equations presents several additional difficulties. First, one needs to identify a suitable multiplier. For uniformly hyperbolic equations, the diffusion is directionally nondegenerate, and standard multipliers can therefore be chosen according to the geometry of the problem. In the degenerate setting, however, the diffusion coefficient may vanish in certain directions or on part of the boundary, and the choice of a suitable multiplier becomes substantially more delicate. Second, the integration-by-parts arguments require additional care. In the degenerate setting, the regularity of weak solutions may not be sufficient to justify the usual Green formula directly. Consequently, appropriate integration-by-parts formulas and boundary terms need to be established in the relevant weighted Sobolev spaces.
	
	In \cite{Yang4}, the authors used a cut-off method to establish controllability when the control region covers the degenerate part. This approach is effective for a broad class of degenerate equations, but it is not directly applicable when the control region is separated from the degenerate part. To overcome this limitation, the authors of \cite{Yang6} introduced an approximation method based on regularized coefficients and established controllability for a degenerate hyperbolic equation with a single-point degeneracy. In \cite{Yang3}, the authors developed a shape design method to establish controllability for several classes of degenerate hyperbolic equations. In these works, the control region is required to be located away from the degenerate part, while the underlying domain is assumed to satisfy certain geometric properties. These geometric assumptions, however, arise naturally from the control problem and are not artificially imposed. In \cite{Yang2}, the authors employed a Carleman estimate to establish an observability inequality for a two-dimensional degenerate hyperbolic equation. In the one-dimensional setting, this is a natural and effective approach. However, for more general degenerate hyperbolic equations, especially in higher dimensions, the Carleman approach becomes considerably more delicate. The degeneracy of the principal part affects the construction of the Carleman weight, and the standard choice
	$
	\psi(x,t)=|x-x_0|^2-\beta |t-t_0|^2+\beta_0
	$
	may no longer be suitable for the degenerate operator. Consequently, obtaining a Carleman estimate adapted to the degeneracy and geometry of the equation is a nontrivial issue. 
	
	Shape design and shape optimization methods have been widely used in various areas of mathematics, industry, and engineering; see, for example, \cite{Buttazzo,Chenais,Greco,Guo2,Guo1,Guo,He,Henrot,Privat,Tiba,Wang}. In particular, we applied the shape design method to the controllability problem for a one-dimensional parabolic equation in \cite{Guo}, and to several classes of degenerate hyperbolic equations in \cite{Yang3}.
	
	In the present work, we aim to develop an approximation-based approach to establish controllability for a degenerate hyperbolic equation when the control region  intersecting the degenerate set. More precisely, we first consider a family of uniformly nondegenerate approximating equations and establish observability inequalities whose constants are independent of the approximation parameter. We then pass to the degenerate limit and obtain the corresponding observability inequality for the original equation. Finally, the controllability result follows from the HUM framework.

	In this paper, we study the following degenerate wave equation:
	\begin{equation} \label{01.08.1}
	\begin{cases}
		\partial_{tt}y-\Div(A\nabla y)=0, & \text{in } Q,\\
		y=0, & \text{on } \Sigma,\\
		y(0)=y^0,  \partial_t y(0)=y^1, & \text{in } \Omega,
	\end{cases}
	\end{equation}
	where $
	\Omega=(-1,1)\times(0,1)$, 
	and $
	A=\operatorname{diag}(1,x_2^\alpha), \alpha\in(0,1) $ 
	is a fixed constant. Let $Q=\Omega\times(0,T)$ for some $T>0$, and let $\Sigma=\partial\Omega\times(0,T)$. The initial data satisfy
	$
	y^0\in H_0^1(\Omega;w),  y^1\in L^2(\Omega)$, 
	where $w=x_2^\alpha$ and the space  $H_0^1(\Om;w)$ will be introduced in Section \ref{S2}. We denote 
	\begin{equation*} 
	\mathcal{A}u=-\Div(A\nabla u),
	\end{equation*} 
	and define the conormal derivative by
	\begin{equation*} 
	\frac{\partial y}{\partial\nu_A}=A\nabla y\cdot\nu,
	\end{equation*} 
	where $\nu$ denotes the outward normal vector on $\partial\Omega$. Moreover, we introduce the notation
	\begin{equation*} 
	\Gamma_1^{-1}={-1}\times(0,1),\quad
	\Gamma_1^{1}={1}\times(0,1),
	\end{equation*} 
	and
	\begin{equation*} 
	\Gamma_2^i=(-1,1)\times{i},\quad i=0,1.
	\end{equation*} 
	
	The paper is organized as follows. In Section \ref{S2}, we introduce the solution spaces, namely the weighted Sobolev spaces, and establish the spectral properties of the operator $\mathcal{A}$. Then, by using the separation of variables method, we construct weak solutions of \eqref{01.08.1}. In Section \ref{S3}, we introduce the approximation method based on shape design. More precisely, we approximate the degenerate equation \eqref{01.08.1} by a family of uniformly hyperbolic equations \eqref{08.04.7}. This requires spectral approximation results, which allow us to obtain the convergence of the corresponding weak solutions. In Section \ref{S4}, after establishing an observability inequality for the approximating equation \eqref{08.04.7}, we pass to the limit and obtain the observability inequality for the original degenerate equation \eqref{01.08.1}.
	
	The main results of this paper are Theorem \ref{08.05.T1}, Corollary \ref{08.07.C2}, and Theorems \ref{08.06.T1} and \ref{08.06.T2}.

\section{Preliminary}\label{S2}

In this section, we first introduce the weighted Sobolev spaces (i.e., solution spaces) associated with equation \eqref{01.08.1}. We then present the spectral properties of the partial differential operator $\mcA$. Finally, we introduce the method of separation of variables.

\subsection{Solution spaces}

We first introduce the weighted Sobolev spaces associated with equation \eqref{01.08.1}. Define
\begin{equation*}
	H^1(\Omega;w)
	=
	\left\{
	u\in L^2(\Omega):
	\int_\Omega \nabla u\cdot A\nabla u\,\mathrm{d}x<+\infty
	\right\}.
\end{equation*}
The inner product and norm on $H^1(\Omega;w)$ are defined by
\begin{equation*}
	(u,v)_{H^1(\Omega;w)}
	=
	\int_\Omega
	\left(
	uv+\nabla u\cdot A\nabla v
	\right)\,\mathrm{d}x,
	\qquad
	\|u\|_{H^1(\Omega;w)}
	=
	(u,u)_{H^1(\Omega;w)}^\f{1}{2},
\end{equation*}
respectively.

Define
\begin{equation*}
	H_0^1(\Omega;w)
	=
	\overline{C_0^\infty(\Omega)}^{\,H^1(\Omega;w)},
\end{equation*}
and let
\begin{equation*}
	H^{-1}(\Omega;w)
\end{equation*}
denote the dual space of $H_0^1(\Omega;w)$ with respect to the pivot space $L^2(\Omega)$.

Next, define
\begin{equation*}
	H^2(\Omega;w)
	=
	\left\{
	u\in L^2(\Omega):
	\mathcal A u\in L^2(\Omega)
	\right\}.
\end{equation*}
The inner product and norm on $H^2(\Omega;w)$ are defined by
\begin{equation*}
	(u,v)_{H^2(\Omega;w)}
	=
	(u,v)_{H^1(\Omega;w)}
	+
	(\mathcal A u,\mathcal A v)_{L^2(\Omega)},
	\qquad
	\|u\|_{H^2(\Omega;w)}
	=
	(u,u)_{H^2(\Omega;w)}^\f{1}{2},
\end{equation*}
respectively.

Finally, define the domain of $\mathcal A$ by
\begin{equation*}
	D(\mathcal A)
	=
	H_0^1(\Omega;w)\cap H^2(\Omega;w).
\end{equation*}

The following result is well known (see, e.g., \cite{Fabes,GC,Heinonen}).

\begin{lemma}\label{08.04.L1}
	The spaces 
	\begin{equation*}
		H_0^1(\Om;w),\quad H^1(\Om;w), \mbox{ and } H^2(\Om;w) 
	\end{equation*}
	are Hilbert spaces. 
\end{lemma}

The following lemma is the weighted Hardy inequality that will be used throughout this paper. 

\begin{lemma}\label{08.04.L2}
	There exists a positive constant $C$, depending only on $\al$, such that 
	\begin{equation*}
		\int_\Om x_2^{\al-2}u^2\df x\leq C\int_\Om x_2^\al (\pt_{x_2}u)^2\df x
	\end{equation*}
	for all $u\in H_0^1(\Om;w)$. Moreover, we have 
	\begin{equation*}
		\int_\Om x_2^{\al-2}u^2\df x\leq C\int_\Om \nabla u\cdot A\nabla u\df x. 
	\end{equation*}
\end{lemma}

\begin{proof}
	By a density argument, it suffices to prove the result for
	\begin{equation*}
	u\in C_0^\infty(\Omega).
	\end{equation*}
	
	Indeed, for each $\be\in (\al,1)$, we have 
	\begin{equation*}
		\begin{split}
			\int_\Om x_2^{\al-2}u^2\df x
			&=\int_{-1}^1\int_0^1 x_2^{\al-2}u^2(x_1,x_2)\df x_2\df x^1=\int_{-1}^1\int_0^1 x_2^{\al-2}\left(\int_0^{x_2} \pt_{x_2}u(x_1,s)\df s\right)^2\df x_2\df x^1\\
			&\leq \int_{-1}^1\int_0^1x_2^{\al-2}\left(\int_0^{x_2} s^\be |\pt_{x_2}u(x_1,s)|^2\df s\right)\left(\int_0^{x_2} s^{-\be}\df s\right)\df x_2\df x_1\\
			&=\f{1}{1-\be}\int_{-1}^1\int_0^1 \int_0^{x_2}  x_2^{\al-\be-1}s^\be |\pt_{x_2}u(x_1,s)|^2\df s\df x_2\df x_1\\
			&=\f{1}{1-\be}\int_{-1}^1\int_0^1\int_{s}^1x_2^{\al-\be-1}s^\be |\pt_{x_2}u(x_1,s)|^2\df x_2\df s\df x_1\\
			&\leq \f{1}{(\be-\al)(1-\be)}\int_{-1}^1\int_0^1 s^\al |\pt_{x_2}u(x_1,s)|^2\df s\df x_1
		\end{split}
	\end{equation*}
	by 
	\begin{equation*}
		\begin{split}
			\int_s^1x_2^{\al-\be-1}\df x_2=\f{1}{\al-\be}\left(1-{s^{\al-\be}}\right)=\f{1}{\be-\al}\left({s^{\al-\be}}-1\right)\leq \f{1}{\be-\al}s^{\al-\be}, 
		\end{split}
	\end{equation*}
	and then 
	\begin{equation*}
		\int_\Om x_2^{\al-2}u^2\df x\leq \f{4}{(1-\al)^2}\int_\Om x_2^\al |\pt_{x_2}u|^2\df x 
	\end{equation*}
	by taking $\be=\f{1+\al}{2}$. This completes the proof.
\end{proof}

\begin{remark}\label{08.04.R1}
	From Lemma \ref{08.04.L2}, we get 
	\begin{equation*}
		\begin{split}
			\int_\Om u^2\df x=\int_\Om x_2^{2-\al}x_2^{\al-2}u^2\df x\leq \int_\Om x_2^{\al-2}u^2\df x\leq C\int_\Om \nabla u\cdot A\nabla u\df x, 
		\end{split}
	\end{equation*}
	hence, the norm 
	\begin{equation*}
		\|u\|_{H_0^1(\Om;w)}=\left(\int_\Om \nabla u\cdot A\nabla u\df x\right)^\f{1}{2}
	\end{equation*}
	defines an equivalent norm on  $H_0^1(\Om;w)$. Throughout this paper,
	we use this equivalent norm on  $H_0^1(\Om;w)$.
\end{remark}

\begin{lemma}\label{08.04.L3}
	The embedding $H_0^1(\Om;w)\hra L^2(\Om)$ is compact. 
\end{lemma}

\begin{proof}
	Let $\{u_n\}_{n\in\N^*}\s H_0^1(\Om;w)$ be a bounded sequence, i.e., there exists a positive constant $M>0$ such that $\|u_n\|_{H_0^1(\Om;w)}\leq M$. Then there exists a subsequence of $\{u_n\}_{n\in\N^*}$, still denoted by itself, and $u_0\in H_0^1(\Om;w)$ such that 
	\begin{equation*}
		u_n\ra u_0 \mbox{ weakly in }H_0^1(\Om;w), \mbox{ and } u_n\ra u_0 \mbox{ weakly in }L^2(\Om). 
	\end{equation*}
	Replacing
	$u_n$
	by
	$u_n-u_0$, it suffices to assume that $u_0=0$.

	Let $\e>0$. For each $\de\in (0,\f{1}{4})$, denote 
	\begin{equation}\label{08.04.1}
		\Om_\de=(-1,1)\ts (\de,1) \mbox{ for all } \de\in (0,1). 
	\end{equation} 
	From Lemma \ref{08.04.L2}, we have 
	\begin{equation*}
		\begin{split}
			\int_{\Om_\de^c} u_n^2\df x
			&=\int_{\Om_\de^c} x_2^{2-\al}x_2^{\al-2}u_n^2\df x\leq \de^{2-\al}\int_\Om x_2^{\al-2}u_n^2\df x\leq C\de^{2-\al}\int_\Om \nabla u_n\cdot A\nabla u_n\df x\leq CM^2\de^{2-\al}. 
		\end{split}
	\end{equation*}
	Choosing  $\de_0>0$ such that for all $\de\in (0,\de_0]$  we have $\|u_n\|_{L^2(\Om_\de^c)}\leq \f{1}{2}\e$. Note that  
	\begin{equation*}
		\int_{\Om_{\de_0}} |\nabla u_n|^2\df x\leq \int_{\Om_{\de_0}}\left(|\pt_{x_1}u_n|^2+x_2^{-\al}x_2^\al |\pt_{x_2}u_n|^2\right)\df x\leq \de_0^{-\al}\int_\Om \nabla u\cdot A\nabla u\df x\leq M^2\de_0^{-\al}, 
	\end{equation*}
	which shows that $\{u_n|_{\Om_{\de_0}}\}_{n\in\N^*}\s  H^1(\Om_{\de_0})$ is bounded. Since the embedding $H^1(\Om_{\de_0})\hra L^2(\Om_{\de_0})$ is compact, then there exists a subsequence of $\{u_n|_{\Om_{\de_0}}\}_{n\in\N^*}$, still denoted by itself, such that 
	\begin{equation*}
		u_n|_{\Om_{\de_0}}\ra 0 \mbox{ strongly in } L^2(\Om_{\de_0}). 
	\end{equation*}
	Choose
	$n_\e$
	sufficiently large
	such that, such that $\|u_{n_\e}\|_{L^2(\Om_{\de_0})}\leq \f{1}{2}\e$, then  $\|u_{n_\e}\|_{L^2(\Om)}\leq \|u_{n_\e}\|_{L^2(\Om_{\de_0}^c)}+\|u_{n_\e}\|_{L^2(\Om_{\de_0})}<\e$.  This completes the proof.
\end{proof}

\subsection{Spectrum}

\begin{notation}\label{08.04.N1}
By Lemma \ref{08.04.L3} and the spectral theorem for compact
self-adjoint operators,
the operator $\mathcal A$ has a purely discrete spectrum
\begin{equation*}
0<\lambda_1\le\lambda_2\le\la_3\leq \cdots,\quad
\lambda_n\to\infty \mbox{ as } n\ra\iy.
\end{equation*}
i.e., associated with each eigenvalue $\la_n\ (n\in\N^*)$, there exists at least one $n$th eigenfunction $\Phi_n\in D(\mcA)$ such that 
\begin{equation}\label{08.04.2}
	\begin{cases}
		\mcA \Phi_n=\la_n \Phi_n, &x\in \Om, \\
		\Phi_n=0, &x\in\pt\Om. 
	\end{cases}
\end{equation}
We choose  $\{\Phi_n\}_{n\in\N^*}$ the orthonormal basis of $L^2(\Om)$. Moreover,  $\{\la_n^{-\f{1}{2}}\Phi_n\}_{n\in\N^*}$ is an orthonormal basis of $H_0^1(\Om;w)$. 
Hence, if  $u\in L^2(\Om)$, then $u=\sum_{n=1}^\iy (u,\Phi_n)_{L^2(\Om)}\Phi_n$. \end{notation}

\begin{definition}\label{08.04.D1}
	Let $\theta\in [0,+\iy)$. Define 
	\begin{equation*}
		D(\mcA^\theta)=\left\{u\in L^2(\Om)\colon \sum_{n=1}^\iy \la_n^{2\theta} (u, \Phi_n)_{L^2(\Om)}^2<+\iy\right\}.
	\end{equation*}
	The inner product and norm on $D(\mcA^\theta)$ are defined by 
	\begin{equation*}
		(u,v)_{D(\mcA^\theta)}=\sum_{n=1}^\iy \la_n^{2\theta} (u,\Phi_n)_{L^2(\Om)}(v,\Phi_n)_{L^2(\Om)},  
	\end{equation*}
	and 
	\begin{equation*}
		\|u\|_{D(\mcA^\theta)}=\left(\sum_{n=1}^\iy \la_n^{2\theta}(u,\Phi_n)_{L^2(\Om)}^2\right)^\f{1}{2}
	\end{equation*}
	respectively. 
\end{definition}

\begin{lemma}\label{08.04.L7}
	Let $\theta\in [0,+\iy)$. The space $(D(\mcA^\theta), (\cdot, \cdot)_{D(\mcA^\theta)})$ is a Hilbert space. Moreover, $D(\mcA^\f{1}{2})=H_0^1(\Om;w)$ and $D(\mcA^1)=D(\mcA)$. 
\end{lemma}

\begin{proof}
	The first assertion follows directly from the spectral representation of
	$\mathcal A$. Indeed, the mapping
	\begin{equation*}
	u\mapsto \{\lambda_n^\theta (u,\Phi_n)_{L^2(\Omega)}\}_{n\ge1}
	\end{equation*}
	is an isometry from $D(\mathcal A^\theta)$ into the  Hilbert space  $\ell^2$, and hence
	$D(\mathcal A^\theta)$ is a Hilbert space.
	
	Moreover, by the spectral characterization of the domain of a
	self-adjoint operator,
	\begin{equation*}
	D(\mathcal A)
	=
	\left\{
	u\in L^2(\Omega):
	\sum_{n=1}^{\infty}\lambda_n^2
	(u,\Phi_n)_{L^2(\Omega)}^2<\infty
	\right\}
	=
	D(\mathcal A^1).
	\end{equation*}
	
	Finally, $\mathcal A$ is the operator associated with the closed symmetric bilinear form 
	\begin{equation*}
	a(u,v)=\int_\Omega w\nabla u\cdot\nabla v\,dx
	\end{equation*}
	on $H_0^1(\Omega;w)$. Hence, by the Kato's first representation theorem,  for
	closed forms,
	\begin{equation*}
	D(\mathcal A^{\f{1}{2}})=D(a)=H_0^1(\Omega;w).
	\end{equation*}
	This completes the proof. 
\end{proof}

\begin{corollary}\label{08.04.C1}
	Let $\theta\in [1,+\iy)$. Then the operator 
	\begin{equation*}
		\mcA: D(\mcA^\theta)\ra D(\mcA^{\theta-1})
	\end{equation*}
	is an isometry.
\end{corollary}

\begin{proof}
	For $u=\sum_{n=1}^\iy (u,\Phi_n)_{L^2(\Om)}\Phi_n$, we obtain 
	\begin{equation*}
		\mcA u=\sum_{n=1}^\iy (u, \Phi_n)_{L^2(\Om)}\mcA \Phi_n=\sum_{n=1}^\iy \la_n(u,\Phi_n)_{L^2(\Om)}\Phi_n, 
	\end{equation*}
	and hence
	\begin{equation*}
		\|\mcA u\|_{D(\mcA^{\theta-1})}^2=\sum_{n=1}^\iy \la_n^{2(\theta-1)}\la_n^2(u,\Phi_n)_{L^2(\Om)}^2=\sum_{n=1}^\iy \la_n^{2\theta}(u,\Phi_n)_{L^2(\Om)}^2=\|u\|_{D(\mcA^\theta)}^2. 
	\end{equation*}
	This proves the corollary.
\end{proof}

\subsection{Method of separation of variables}

\begin{notation}\label{08.04.N2}
	Let 
	\begin{equation}\label{08.04.6}
		\La_1=(-1,1), \mbox{ and }\La_2=(0,1). 
	\end{equation}
	
	Denote 
	\begin{equation}\label{08.04.3}
		\mcA^{(1)}u=-\pt_{x_1x_1}u,\quad \mcA^{(2)}u=-\pt_{x_2}(x_2^\al \pt_{x_2}u), 
	\end{equation}
	then 
	\begin{equation}\label{08.04.4}
		\mcA u=\mcA^{(1)}u+\mcA^{(2)}u. 
	\end{equation}
	
\end{notation}

The weighted Sobolev spaces
$H^1(\Lambda_2;w)$,
$H_0^1(\Lambda_2;w)$
and
$H^2(\Lambda_2;w)$
are defined in the same way as those on $\Omega$.
For example,
\begin{equation*}
	H^2({\La_2};w)=\left\{u\in H^1({\La_2};w)\colon \mcA^{(2)}u\in L^2(\La_2)\right\}, 
\end{equation*}
the inner product and norm on $H^2({\La_2};w)$ are defined by 
\begin{equation*}
	(u,v)_{H^2({\La_2};w)}=(u,v)_{H^1({\La_2};w)}+(\mcA^{(2)}u,\mcA^{(2)}v)_{L^2({\La_2})},\mbox{ and } \|u\|_{H^2({\La_2};w)}=(u,u)_{H^2({\La_2};w)}^\f{1}{2}. 
\end{equation*}

Define 
\begin{equation}\label{08.04.5}
	D(\mcA^{(1)})=H_0^1({\La_1})\cap H^2({\La_1}), \mbox{ and } D(\mcA^{(2)})=H_0^1({\La_2};w)\cap H^2({\La_2};w). 
\end{equation}

The one-dimensional weighted Hardy inequality below 
will be used repeatedly.

\begin{lemma}[{\cite[(2.1), p.~165]{Alabau}}]\label{08.04.L4}
	There exists a positive constant $C$, depending only on $\al$, such that 
	\begin{equation*}
		\int_\Om x_2^{\al-2}u^2\df x\leq C\int_\Om x_2^\al (\pt_{x_2}u)^2\df x
	\end{equation*}
	for all $u\in H_0^1({\La_2};w)$.
\end{lemma}

\begin{remark}\label{08.04.R2}
	From Lemma \ref{08.04.L4}, we get 
	\begin{equation*}
		\begin{split}
			\int_{\La_2} u^2\df x_2=\int_{\La_2} x_2^{2-\al}x_2^{\al-2}u^2\df x_2\leq \int_{\La_2} x_2^{\al-2}u^2\df x_2\leq C\int_{\La_2} w(\pt_{x_2}u)^2\df x_2, 
		\end{split}
	\end{equation*}
	hence, the norm 
	\begin{equation*}
		\|u\|_{H_0^1({\La_2};w)}=\left(\int_{\La_2} w(\pt_{x_2}u)^2\df x\right)^\f{1}{2}
	\end{equation*}
	is an equivalent norm of $H_0^1({\La_2};w)$. In what follows, we use this norm on $H_0^1({\La_2};w)$. 
\end{remark}

\begin{lemma}\label{08.04.L5}
	The embedding $H_0^1({\La_2};w)\hra L^2({\La_2})$ is compact. 
\end{lemma}

\begin{proof}
	The proof of this lemma is the same as the proof of Lemma \ref{08.04.L3}. 
\end{proof}

\begin{notation}\label{08.04.N3}
	Since the embeddings 
	\begin{equation*}
		H_0^1({\La_1})\hra L^2({\La_1}), \mbox{ and } H_0^1({\La_2};w)\hra L^2({\La_2}) 
	\end{equation*}
	are compact,  
	the operators
	$\mathcal A^{(1)}$
	and
	$\mathcal A^{(2)}$
	have purely discrete simple spectra 
	\begin{equation*}
		0<\la_1^{(i)}\leq \la_2^{(i)}\leq \la_3^{(i)}\leq \cdots \ra +\iy, \ i=1,2, 
	\end{equation*}
	where $\mcA^{(1)}, \mcA^{(2)}$ are defined in \eqref{08.04.3}. Associated with each eigenvalue
	$\lambda_n^{(i)}$
	there exists an eigenfunction
	$\Phi_n^{(i)}$
	satisfying
	\begin{equation*}
		\begin{cases}
			\mcA^{(i)}\Phi_n^{(i)}=\la_n^{(i)}\Phi_n^{(i)}, &x\in {\La_i}, \\
			\Phi_n^{(i)}=0, &x\in \pt {\La_i}. 
		\end{cases} 
	\end{equation*}
	We denote $\{\Phi_n^{(i)}\}_{n\in\N^*}$  the orthonormal basis of $L^2(\La_i)$. Moreover,  $\{(\la_n^{(1)})^{-\f{1}{2}}\Phi_n^{(1)}\}_{n\in\N^*}$ is an orthonormal basis of $H_0^1({\La_1})$, and $\{(\la_n^{(2)})^{-\f{1}{2}}\Phi_n^{(2)}\}_{n\in\N^*}$ is an orthonormal basis of $H_0^1({\La_2})$. 
\end{notation}

From \cite[Lemma 2, p.~2047]{Gueye}, or  \cite[(74), p.~198]{Cannarsa}, or  \cite[(1.3), p.~2509]{Buffe}, there exists a positive constant $C$ such that 
\begin{equation}\label{08.05.7}
	\sqrt{\la_{k+1}^{(2)}}-\sqrt{\la_k^{(2)}}\geq C(2-\al), \mbox{ for all }k\in\N^* \mbox{ and } \al\in [0,1). 
\end{equation}

The following lemma follows from the method of separation of variables.

\begin{lemma}[{\cite[Lemma 3.3]{Yang}}]\label{08.04.L6}
	We have 
	\begin{equation*}
		H_0^1(\Om;w)=H_0^1({\La_1})\wh\ots H_0^1({\La_2};w). 
	\end{equation*} 
\end{lemma}

\begin{remark}\label{08.04.R3}
	From Lemma \ref{08.04.L6}, we obtain  
	$\{\Phi_n^{(1)}\Phi_k^{(2)}\}_{n,k\in\N^*}$ is an orthonormal basis of $L^2(\Om)$, and $\la_n^{(1)}+\la_k^{(2)}\ (n,k\in\N^*)$  form the spectrum  $\mcA$. Hence, for all $m\in\N^*$, we have the form $\la_m=\la_n^{(1)}+\la_k^{(2)}$, and $\Phi_m=\Phi_n^{(1)}\Phi_k^{(2)}$ for some $n,k\in\N^*$.  Throughout in this paper, we denote
	\begin{equation*}
		\la_{nk}
		=
		\lambda_n^{(1)}+\lambda_k^{(2)}.
	\end{equation*}  
\end{remark}

The counterparts of Lemma
\ref{08.04.L7}
and Corollary
\ref{08.04.C1}
remain valid for the operators
$\mathcal A^{(1)}$
and
$\mathcal A^{(2)}$.

\begin{remark}\label{08.05.R1}
	The eigenpairs of
	$\mathcal A^{(1)}$
	are explicitly given by 
	\begin{equation*}
		\la_n^{(1)}=\f{n^2\pi^2}{4},\quad \Phi_n^{(1)}=\sin \f{n\pi(x_1+1)}{2}. 
	\end{equation*}
	The eigenpairs of
	$\mathcal A^{(2)}$
	are explicitly given by 
	\begin{equation*}
		\la_n^{(2)}=(\ka j_{\nu,n})^2, \quad \Phi_n^{(2)}=\f{(2\ka)^\f{1}{2}}{|J_\nu'(j_{\nu,n})|}x_2^{\f{1}{2}(2-\al)}J_\nu (j_{\nu, n}x_2^\ka), 
	\end{equation*}
	where $\nu=\f{1-\al}{2-\al}$ and $\ka=\f{1}{2}(2-\al)$. 
	(See \cite[(4.35) and (4.38), p.~2046-2047]{Gueye}.) 
	
	It is clear that $\lambda_n^{(1)}\sim n^2$ and, by the
	asymptotic behavior of the zeros of the Bessel functions,
	$\lambda_n^{(2)}\sim n^2$ as $n\to\infty$.
\end{remark}

\subsection{Weak solution}

Let
\begin{equation*}
y^0\in H_0^1(\Omega;w),\qquad
y^1\in L^2(\Omega).
\end{equation*}
By Remark \ref{08.04.R3},  the initial data admit the expansions
\begin{equation*}
y^0=\sum_{n,k=1}^\infty y_{nk}^0\Phi_n^{(1)}\Phi_k^{(2)},
\qquad
y^1=\sum_{n,k=1}^\infty y_{nk}^1\Phi_n^{(1)}\Phi_k^{(2)},
\end{equation*}
where the Fourier coefficients are given by
\begin{equation*}
	y_{nk}^0=(y^0, \Phi_n^{(1)}\Phi_k^{(2)})_{L^2(\Om)}, \mbox{ and } y_{nk}^1=(y^1,\Phi_n^{(1)}\Phi_k^{(2)})_{L^2(\Om)}, \mbox{ for all } n,k\in\N^*. 
\end{equation*}
The weak solution to \eqref{01.08.1} is therefore given by
\begin{equation}\label{08.05.11}
	y=\sum_{n,k=1}^\iy \left(y_{nk}^0 \cos\left(\sqrt{\la_{nk}}t\right)+y_{nk}^1\f{1}{\sqrt{\la_{nk}}}\sin\left(\sqrt{\la_{nk}}t\right)\right)\Phi_n^{(1)}\Phi_k^{(2)}, 
\end{equation}
where $\la_{nk}\ (n,k\in\N^*)$ is defined in Remark \ref{08.04.R3}.

\section{Shape design}\label{S3}

In this section, we shall show that the weak solution of \eqref{01.08.1} can be approximated by a family of wave equations with uniformly elliptic operators. 

Let $\de\in (0,\f{1}{4})$. We consider the following uniformly hyperbolic equation
\begin{equation}\label{08.04.7}
	\begin{cases}
		\pt_{tt}y_\de-\Div(A_\de \nabla y_\de)=0, &\mbox{in }Q_\de, \\
		y_\de=0, &\mbox{on }\Sigma_\de,\\
		y_\de(0)=y_\de^0, \pt_ty_\de(0)=y_\de^1, &\mbox{in }\Om_\de, 
	\end{cases}
\end{equation}
where $
\Omega_\delta=(-1,1)\times(\delta,1)$, 
and $
Q_\delta=\Omega_\delta\times(0,T), 
\Sigma_\delta=\partial\Omega_\delta\times(0,T)$. 
Moreover, $
A_\delta=A|_{\Omega_\delta}$ and $w_\de=w|_{\Om_\de}$. 
Denote
\begin{equation*}
	\Ga_{\de,1}^{-1}=\{-1\}\ts (\de,1), \quad \Ga_{\de,1}^1=\{1\}\ts (\de,1), \quad \Ga_{\de,2}^\de=(-1,1)\ts \{\de\},\quad \Ga_{\de,2}^1=(-1,1)\ts \{1\}. 
\end{equation*}

To solve \eqref{08.04.7}, we introduce
\begin{equation*}
	H_0^1(\Om_\de;w_\de), \quad H^1(\Om_\de;w_\de), \mbox{ and } H^2(\Om_\de;w_\de). 
\end{equation*}
For example, 
\begin{equation*}
	H^1(\Om_\de;w_\de)=\left\{u\in L^2(\Om_\de)\colon \int_{\Om_\de} \nabla u\cdot A_\de\nabla u\df x<+\iy \right\}. 
\end{equation*}
Since $x_2\ge\delta>0$ in $\Omega_\delta$, the weighted Sobolev norms are equivalent to the standard Sobolev norms. Consequently,
$
H_0^1(\Omega_\delta;w_\delta)=H_0^1(\Omega_\delta), 
H^1(\Omega_\delta;w_\delta)=H^1(\Omega_\delta).
$
Moreover, by elliptic regularity,
$
H^2(\Omega_\delta;w_\delta)=H^2(\Omega_\delta)$. 

We  also introduce 
\begin{equation*}
	\La_2^\de=(\de,1). 
\end{equation*}
And define $H_0^1(\La_2^\de;w_\de), H^1(\La_2^\de;w_\de)$ and $H^2(\La_2^\de;w_\de)$ as above. Define
\begin{equation*}
	\mcA_\de u=-\Div(A_\de\nabla u), 
\end{equation*}
and 
\begin{equation*}
	\mcA_\de u =\mcA_\de ^{(1)} u+\mcA_\de^{(2)} u, \mbox{ with } \mcA_\de^{(1)}=-\pt_{x_1x_1}u \mbox{ and } \mcA_\de^{(2)}=-\pt_{x_2}(x_2^\al \pt_{x_2}u). 
\end{equation*}

Since the embedding  $H_0^1(\La_2^\de;w_\de)\hra L^2(\La_2^\de)$ is compact, then $\mcA_\de^{(2)}$ has a purely discrete simple spectrum
\begin{equation*}
	0<\la_{\de,1}^{(2)}\leq \la_{\de,2}^{(2)}\leq \la_{\de,3}^{(2)}\leq \cdots \ra+\iy. 
\end{equation*}
And we denote $\Phi_{\de,n}^{(2)}$ the normalized $n$th eigenfunction of $\mcA_\de^{(2)}$ with respect to the $n$th eigenvalue  $\la_{\de,n}^{(2)}$.

From above, we get the following lemma by Lemma \ref{08.04.L6}.  

\begin{lemma}\label{08.04.L9}
	We have 
	\begin{equation*}
		H_0^1(\Om_\de;w_\de)=H_0^1({\La_1})\wh\ots H_0^1({\La_2^\de};w_\de). 
	\end{equation*} 
\end{lemma}

\begin{remark}\label{08.04.R4}
	From Lemma \ref{08.04.L9}, we obtain  
	$\{\Phi_n^{(1)}\Phi_{\de,k}^{(2)}\}_{n,k\in\N^*}$ is an  orthonormal basis of $L^2(\Om)$, and $\la_n^{(1)}+\la_{\de,k}^{(2)}\ (n,k\in\N^*)$ form the spectrum of $\mcA_\de$. Hence, for all $m\in\N^*$, there exist $n,k\in\mathbb N^*$ such that  $\la_m=\la_n^{(1)}+\la_{\de,k}^{(2)}$, and $\Phi_m=\Phi_n^{(1)}\Phi_{\de,k}^{(2)}$ for some $n,k\in\N^*$. Throughout in this paper, we denote
	\begin{equation*}
		\la_{nk}^\de
		=
		\lambda_n^{(1)}+\lambda_{\de,k}^{(2)}.
	\end{equation*} 
\end{remark}

Let
\begin{equation*}
	y_\de^0\in H_0^1(\Omega_\de;w),\qquad
	y_\de^1\in L^2(\Omega_\de).
\end{equation*}
By Remark \ref{08.04.R4}, they admit the expansions
\begin{equation*}
	y_\de^0=\sum_{n,k=1}^\infty y_{\de,nk}^0\Phi_n^{(1)}\Phi_{\de,k}^{(2)},
	\qquad
	y_\de^1=\sum_{n,k=1}^\infty y_{\de,nk}^1\Phi_n^{(1)}\Phi_{\de,k}^{(2)},
\end{equation*}
where 
\begin{equation*}
	y_{\de,nk}^0=(y_\de^0, \Phi_n^{(1)}\Phi_{\de,k}^{(2)})_{L^2(\Om_\de)}, \mbox{ and } y_{\de, nk}^1=(y_\de^1,\Phi_n^{(1)}\Phi_{\de,k}^{(2)})_{L^2(\Om_\de)}, \mbox{ for all } n,k\in\N^*. 
\end{equation*}
The weak solution of \eqref{08.04.7} is given by
\begin{equation}\label{08.05.10}
	\begin{aligned}
		y_\delta(x,t)
		&=
		\sum_{n,k=1}^{\infty}
		\left(
		y_{\delta,nk}^0
		\cos\left(\sqrt{\lambda_{nk}^{\delta}}t\right)
		+y_{\delta,nk}^1
		\frac{1}
		{\sqrt{\lambda_{nk}^{\delta}}}
		\sin\left(\sqrt{\lambda_{nk}^{\delta}}t\right)
		\right)
		\Phi_n^{(1)}(x_1)\Phi_{\delta,k}^{(2)}(x_2).
	\end{aligned}
\end{equation}

We conclude this subsection by noting that, when $\delta=0$, the notation introduced above reduces to that used in Section \ref{S2}. For example, $\mathcal{A}_0=\mathcal{A}$ and $w_0=w$, etc.

\subsection{Spectral approximation of  $\mcA^{(2)}$}

For each $0\leq \de_1\leq \de_2\leq \f{1}{4}$, if $u_{\de_2}\in H_0^1(\La_2^{\de_2};w_{\de_2})$, define 
\begin{equation}\label{08.04.8}
	E_{\de_2}^{\de_1}u_{\de_2}(x)=
	\begin{cases}
		u_{\de_2}(x), &x\in \La_2^{\de_2},\\
		0, &x\in \La_2^{\de_1}-\La_2^{\de_2}, 
	\end{cases} 
\end{equation}
and $E_0^0 u=u$ for all $u\in H_0^1(\Om;w)$. 
The extension $E_{\delta_2}^{\delta_1}u_{\delta_2}$ belongs to
$H_0^1(\Lambda_2^{\delta_1};w_{\delta_1})$,   and $\|E_{\de_2}^{\de_1}u_{\de_2}\|_{H_0^1(\La_2^{\de_1};w_{\de_1})}=\|u_{\de_2}\|_{H_0^1(\La_2^{\de_2};w_{\de_2})}$. 

Now, for each $\de\in [0,\f{1}{4}]$, define
\begin{equation}\label{08.04.11}
	r_{\de}: L^2(\La_2)\ra L^2(\La_2), \ f\mapsto r_\de (f)=E_{\de}^0u_\de^f, 
\end{equation}
where  $u_\delta^f$ denotes the weak solution of 
\begin{equation}\label{08.04.10}
	\begin{cases}
		\mcA_\de^{(2)} u_\de^f=f|_{\La_2^\de}, &x\in \La_2^\de, \\
		u_\de^f=0, &x\in \pt\La_2^\de. 
	\end{cases}
\end{equation}
Then, for each $\de\in [0,\f{1}{4}]$,  from $u_\de^f\in H_0^1(\La_2^\de;w_\de)$, we get $E_\de^0u_\de^f\in H_0^1(\La_2;w)$, and 
\begin{equation*}
	\begin{split}
		\|r_\de(f)\|_{L^2(\La_2)}
		&=\|E_\de^0 u_\de^f\|_{L^2(\La_2)}\leq C\|E_\de^0u_\de^f\|_{H_0^1(\La_2;w)}= C\|u_\de^f\|_{H_0^1(\La_2^\de;w_\de)}\leq  C\|f\|_{L^2(\La_2)}
	\end{split}
\end{equation*}
by Remark \ref{08.04.R2} and $\|u_\de ^f\|_{H_0^1(\La_2^\de;w_\de)}\leq C\|f|_{\La_2^\de}\|_{L^2(\La_2^\de)}\leq C\|f\|_{L^2(\La_2)}$, where the positive constants $C$ depends only on $\al$. Moreover, since
$
r_\delta:L^2(\Lambda_2)\rightarrow L^2(\Lambda_2)
$
is bounded and the embedding
$
H_0^1(\Lambda_2;w)\hookrightarrow L^2(\Lambda_2)
$
is compact, $r_\delta$ is compact. Hence, $r_\de\ (0\leq \de\leq \f{1}{4})$ is a compact, self-adjoint and nonnegative operator on $L^2(\La_2)$.

\begin{lemma}\label{08.04.L12}
	Let $f\in L^2(\Om)$. We have 
	\begin{equation*}
		r_\de(f)\ra r_0(f) \mbox{ strongly in } H_0^1(\La_2;w) \mbox{ as } \de\ra 0^+. 
	\end{equation*}
\end{lemma}

\begin{proof}
	Testing \eqref{08.04.10} by $u_\delta^f$, we obtain
	\begin{equation*}
		\begin{split}
			\|r_\de(f)\|_{H_0^1(\La_2;w)}^2
			&=\int_{\La_2}w (\pt_{x_2} E_\de^0 u_\de^f)^2\df x_2=\int_{\La_2^\de} w_\de (\pt_{x_2}u_\de^ f)^2\df x_2=\int_{\La_2^\de}fu_\de^f\df x_2\\
			&\leq \|f\|_{L^2(\La_2^\de)}\|u_\de^f\|_{L^2(\La_2^\de)}=\|f\|_{L^2(\La_2^\de)}\|E_\de^0u_\de^f\|_{L^2(\La_2)}\\
			&\leq C\|f\|_{L^2(\La_2)}\|E_\de^0 u_\de^f\|_{H_0^1(\La_2;w)}=C\|f\|_{L^2(\La_2)}\|r_\de(f)\|_{H_0^1(\La_2;w)} 
		\end{split}
	\end{equation*}
	by \eqref{08.04.10} and Remark \ref{08.04.R2}, then 
	\begin{equation*}
		\|r_\de(f)\|_{H_0^1(\La_2;w)}\leq C\|f\|_{L^2(\La_2)}, 
	\end{equation*}
	where the positive constants $C$ depends only on $\al$. By the compact embedding in Lemma \ref{08.04.L5},  there exists a subsequence of $\{r_\de(f)\}_{0<\de\leq \f{1}{4}}$, still denoted by itself, and $u_0\in H_0^1(\La_2;w)$ such that 
	\begin{equation}\label{08.04.12}
		r_\de(f)\ra u_0 \mbox{ weakly in } H_0^1(\La_2;w), \mbox{ and } r_\de(f)\ra u_0 \mbox{ strongly in } L^2(\La_2). 
	\end{equation} 
	
	Now, for each $v\in C_0^\iy(\La_2)$, for all $0<\de< \dist(\supp v, \{0\})$, we have 
	\begin{equation}\label{08.05.1}
		\int_{\La_2} w(\pt_{x_2} r_\de(f))(\pt_{x_2}v)\df x_2=\int_{\La_2^\de} w_\de (\pt_{x_2}u_\de^f)(\pt_{x_2}v)\df x_2=\int_{\La_2^\de} fv\df x_2=\int_{\La_2} f v\df x_2. 
	\end{equation}
	Together with this and \eqref{08.04.12}, we obtain 
	\begin{equation*}
		\int_{\La_2} w(\pt_{x_2}u_0)(\pt_{x_2}v)\df x_2=\int_{\La_2}fv\df x_2. 
	\end{equation*}
	Since $C_0^\infty(\Lambda_2)$ is dense in $H_0^1(\Lambda_2;w)$, this identity characterizes the weak solution of \eqref{08.04.10} for $\delta=0$: 
	\begin{equation}\label{08.04.13}
		u_0=r_0(f). 
	\end{equation}
	
	Finally, from the second part of  \eqref{08.04.12} and \eqref{08.05.1} and \eqref{08.04.13}, by density, the weak formulation extends to all 
	$v\in H_0^1(\Lambda_2;w)$, taking $v=r_\delta(f)$ and $v=r_0(f)$ respectively, we obtain  
	\begin{equation*}
		\begin{split}
			\int_{\La_2} w(\pt_{x_2}r_\de(f))^2\df x_2
			&=\int_{\La_2} fr_\de(f)\df x_2\ra \int_{\La_2} fr_0(f)\df x_2=\int_{\La_2} w(\pt_{x_2}r_0(f))^2\df x_2. 
		\end{split}
	\end{equation*}
	Together with this and the first part of  \eqref{08.04.12}, since $H_0^1(\Lambda_2;w)$ is a Hilbert space,  weak convergence together with convergence of norms implies strong convergence,  we obtain 
	\begin{equation*}
		r_\de(f)\ra r_0(f) \mbox{ strongly in } H_0^1(\La_2;w). 
	\end{equation*}
	We complete the proof of this lemma. 
\end{proof}

\begin{remark}
	Lemma \ref{08.04.L12} can be viewed as a resolvent convergence result, which is closely related to the $\Gamma$-convergence of the associated quadratic forms.
\end{remark}

\begin{lemma}[{\cite[Corollaries XI.9.3 and XI.9.4]{Dunford}}]\label{08.04.L11}
	Let $r_1,r_2$ be compact, self-adjoint and nonnegative operators on $H$. For every $m,n\geq 1$ we have 
	\begin{equation*}
		|\xi_n(r_1)-\xi_n(r_2)|\leq \|r_1-r_2\|_{\mcL(H)},  
	\end{equation*}
	where $\xi_n(r_i)$ is the $n$th eigenvalue of $r_i$ for $i=1,2$. 
\end{lemma}

\begin{lemma}\label{08.04.L10}
	For each $k\in\N^*$, we have 
	\begin{equation*}
		\la_{\de,k}^{(2)}\dra  \la_k^{(2)} \mbox{ as } \de\ra 0^+. 
	\end{equation*}
\end{lemma}

\begin{proof}
	We prove this lemma by the following steps. 
	
	{\it Step 1}. 
	It is well-known that 
	\begin{equation}\label{08.04.9}
		\begin{split}
			\la_{\de,k}^{(2)}
			&=\min_{0\neq v\in H_0^1(\La_2^\de;w_\de)\atop (v, \Phi_{\de,i}^{(2)})_{L^2(\La_2^\de)}=0, i=1,\cdots, k-1}R_\de[v]=\min_{E_k\s H_0^1(\La_2^\de;w_\de)\atop \dim E_k=k}\max_{0\neq v\in E_k}R_\de[v], 
		\end{split}
	\end{equation}
	where 
	\begin{equation*}
		R_\de[v]=\f{\int_{\La_2^\de}w_{\de}(\pt_{x_2}v)^2\df x}{\int_{\La_2^\de}v^2\df x }. 
	\end{equation*}

	{\it Step 2}. The monotonicity with respect to the domain inclusion follows from the min-max characterization. i.e., we prove   $\la_{\de_2,k}^{(2)}\leq \la_{\de_1,k}^{(2)}$ for all $0\leq \de_2\leq \de_1\leq \f{1}{4}$, where $\la_{0,k}^{(2)}=\la_k^{(2)}$.  
	
	For each $E_k\s H_0^1(\Om_{\de_1};w_{\de_1})$ with $\dim E_k=k$, we get $E_k\s H_0^1(\Om_{\de_2};w_{\de_2})$ by \eqref{08.04.8}, then 
	\begin{equation}\label{08.07.1}
		\la_{\de_2,k}^{(2)}\leq \la_{\de_1,k}^{(2)}
	\end{equation}
	by \eqref{08.04.9}. 
	
	{\it Step 3}. We prove $\la_{\de,k}^{(2)}$ is continuous at  $\de=0$. i.e., $\la_{\de,k}^{(2)}\ra \la_k^{(2)}$ as $\de\ra 0^+$. 
	
	From Lemma \ref{08.04.L11}, if we have proved 
	\begin{equation}\label{08.05.3}
		\|r_\de-r_0\|_{\mcL(L^2(\La_2))}\ra 0, 
	\end{equation}
	then 
	\begin{equation*}
		\left|\left[\la_{\de,k}^{(2)}\right]^{-1}-\left[\la_{k}^{(2)}\right]^{-1}\right|\leq \|r_\de-r_0\|_{\mcL(L^2(\La_2))}\ra 0. 
	\end{equation*}
	Which shows that 
	\begin{equation*}
		\left|\la_{\de,k}^{(2)}-\la_k^{(2)}\right|\leq \la_{\de,k}^{(2)}\la_k^{(2)}\|r_\de-r_0\|_{\mcL(L^2(\La_2))}\leq \left[\la_{\f{1}{4},k}^{(2)}\right]^2\|r_\de-r_0\|_{\mcL(L^2(\La_2))}\ra 0
	\end{equation*}
	as $\de\ra 0^+$ by Step 2. 
	
	Now, we prove \eqref{08.05.3}. 
	
	We argue by contradiction. Then there exists $\e_0>0$ such that there exists a subsequence $\{r_{\de_n}\}_{n\in\N^*}$ of $\{r_\de\}_{0<\de<\f{1}{4}}$ such that 
	\begin{equation*}
		\|r_{\de_n}-r_0\|_{\mcL(L^2(\La_2))}\geq \e_0. 
	\end{equation*}
	Hence, there exists $f_n \ (n\in\N^*)$ with $\|f_n\|_{L^2(\La_2)}\leq 1$ such that 
	\begin{equation}\label{08.05.5}
		\|r_{\de_n}(f_n)-r_0(f_n)\|_{L^2(\La_2)}\geq \f{1}{2}\e_0.
	\end{equation}
	From the boundedness of $\{f_n\}_{n\in\N^*}$, then there exists a subsequence of $\{f_n\}_{n\in\N^*}$, still denoted by itself, and $f\in L^2(\La_2)$ such that $f_n\ra f$ weakly in $L^2(\La_2)$. Note that 
	\begin{equation*}
		\|r_{\de_n}(f_n)\|_{H_0^1(\La_2;w)}\leq C\|f_n\|_{L^2(\La_2)}\leq C, 
	\end{equation*}
	where the positive constants $C$ depends only on $\al$, then there exists a subsequence of $\{r_{\de_n}(f_n)\}_{n\in\N^*}$, still denoted by itself, and $z_0\in H_0^1(\La_2;w)$ such that 
	\begin{equation}\label{08.05.4}
		r_{\de_n}(f_n)\ra z_0 \mbox{ weakly in } H_0^1(\La_2;w) \mbox{ and } r_{\de_n}(f_n)\ra z_0 \mbox{ strongly in } L^2(\La_2)   \mbox{ as } n\ra\iy. 
	\end{equation}
	For each $v\in C_0^\iy(\La_2)$, when $\de_n\in (0, \dist(\supp v, \{0\}))$, from \eqref{08.05.4}, we have 
	\begin{equation*}
		\begin{split}
		\int_{\La_2} w(\pt_{x_2} z_0) \pt_{x_2}v\df x_2\leftarrow\int_{\La_2}w(\pt_{x_2}r_{\de_n}(f_n))\pt_{x_2}v\df x_2
		&=\int_{\La_2^\de} w_\de (\pt_{x_2}u_n^{f_n})\pt_{x_2}v\df x_2=\int_{\La_2^\de}f_n v\df x_2\\
		&=\int_{\La_2}f_nv\df x_2\ra \int_{\La_2} fv\df x_2, 
		\end{split} 
	\end{equation*}
	hence 
	\begin{equation}\label{08.05.6}
		z_0=u_0^f=r_0(f). 
	\end{equation}
	
	Finally, since  $\|r_0(f_n)\|_{H_0^1(\La_2;w)}\leq C\|f_n\|_{L^2(\La_2)}$ for all $n\in\N$, where the positive constant $C$ depending only on $\al$, there exists a subsequence of $\{r_0(f_n)\}_{n\in\N^*}$, still denoted by itself, such that 
	\begin{equation*}
		r_0(f_n)\ra r_0(f) \mbox{ weakly in } H_0^1(\La_2;w) \mbox{ and } r_0(f_n)\ra r_0(f) \mbox{ strongly in }L^2(\La_2)  \mbox{ as } n\ra\iy
	\end{equation*}
	by $f_n\ra f$ weakly in $L^2(\La_2)$ and 
	\begin{equation*}
		\int_{\La_2} w (\pt_{x_2}r_0(f))\pt_{x_2}v\df x\leftarrow\int_{\La_2}w(\pt_{x_2} r_0(f_n))\pt_{x_2}v\df x=\int_{\La_2} f_n v\df x_2\ra \int_{\La_2} fv\df x_2. 
	\end{equation*}
	Combining with this and \eqref{08.05.4}  and \eqref{08.05.6}, we obtain  
	\begin{equation*}
		\|r_{\de_n}(f_n)-r_0(f_n)\|_{L^2(\Lambda_2)}
		\leq
		\|r_{\de_n}(f_n)-r_0(f)\|_{L^2(\Lambda_2)}
		+
		\|r_0(f)-r_0(f_n)\|_{L^2(\Lambda_2)}
		 \rightarrow0, 
	\end{equation*}
	this contradicts \eqref{08.05.5}. 
\end{proof}

\begin{lemma}\label{08.05.L1}
	For each $k\in\N^*$, we have 
	\begin{equation*}
		E_\de^0\Phi_{\de,k}^{(2)}\ra \Phi_k ^{(2)} \mbox{ strongly in } H_0^1(\La_2;w) \mbox{ as }\de\ra 0^+. 
	\end{equation*}
\end{lemma}

\begin{proof}
	Since   $\Lambda_2^\delta=(\delta,1)$  and  $x_2^\alpha>0$
	on  $\overline{\Lambda_2^\delta}$ for all $\de\in (0,\f{1}{4}]$, 
	the one-dimensional Sturm-Liouville theory implies that 
	$\lambda_{\delta,k}^{(2)}$
	is simple.
	
	From \eqref{08.07.1}, we obtain 
	\begin{equation*}
		\int_{\La_2} w\left(\pt_{x_2}E_\de^0\Phi_{\de,k}^{(2)}\right)^2\df x_2 =\int_{\La_2^\de} w_\de \left(\pt_{x_2}\Phi_{\de,k}^{(2)}\right)^2\df x_2=\la_{\de,k}^{(2)}\leq \la_{\f{1}{4},k}^{(2)}, 
	\end{equation*}
	there exist a subsequence $\{\de_n\}_{n\in\mathbb N^*}$ with $\de_n\to0^+$ and a function $u_k\in H_0^1(\Lambda_2;w)$ such that
	\begin{equation}\label{08.05.8}
		\begin{split} 
		E_\de^0\Phi_{\de_n,k}^{(2)}
		&\ra u_k \mbox{ weakly in } H_0^1(\La_2;w), \mbox{ and }\\ E_\de^0\Phi_{\de_n,k}^{(2)}
		&\ra u_k \mbox{ strongly in } L^2(\La_2),  \mbox{ as } n\ra\iy.
		\end{split}  
	\end{equation}

	Let $v\in C_0^\infty(\Lambda_2)$. Since $\supp v$ is compactly contained in $\Lambda_2$, there exists $\delta_v>0$ such that $\supp v\subset \Lambda_2^\de$ for every $0<\de<\delta_v$, from Lemma \ref{08.04.L10} and the second part of \eqref{08.05.4}, we obtain 
	\begin{equation*}
		\begin{split} 
		\int_{\La_2}w(\pt_{x_2}u_k)\pt_{x_2}v\df x_2\leftarrow &\int_{\La_2} w(\pt_{x_2}E_\de^0\Phi_{\de,k}^{(2)})(\pt_{x_2}v)\df x_2=\int_{\La_2^\de} w_\de (\pt_{x_2}\Phi_{\de,k}^{(2)})(\pt_{x_2}v)\df x_2\\
		&=\la_{\de,k}^{(2)}\int_{\La_2^\de}\Phi_{\de,k}^{(2)}v\df x_2=\la_{k,\de}^{(2)}\int_{\La_2}E_\de^0\Phi_{\de,k}^{(2)}v\df x_2\ra \la_k^{(2)}\int_{\La_2} u_kv\df x_2. 
		\end{split} 
	\end{equation*}
	This shows that $u_k$ is the eigenfunction of $\mcA^{(2)}$ with respect to the  $k$th eigenvalue $\la_k^{(2)}$. Note that $\la_k^{(2)}$ is simple (see \eqref{08.05.7}), hence
	\begin{equation}\label{08.05.9}
		u_k=\Phi_{k}^{(2)}. 
	\end{equation} 
	
	Finally, from Lemma \ref{08.04.L10}, we have 
	\begin{equation*}
		\int_{\La_2}w(\pt_{x_2}E_{\de_n}^0\Phi_{\de_n,k}^{(2)})^2\df x_2=\la_{\de_n,k}^{(2)}\ra \la_k^{(2)}= \int_{\La_2} w(\pt_{x_2}\Phi_k^{(2)})^2\df x_2.
	\end{equation*}
	Combining the weak convergence in $H_0^1(\Lambda_2;w)$ with the convergence of the corresponding energy norms (see the first part in \eqref{08.05.8}), we conclude that
	\begin{equation*}   E_{\de_n}^0\Phi_{\de_n,k}^{(2)}
	\to
	\Phi_k^{(2)}  \text{ strongly in }H_0^1(\Lambda_2;w).
\end{equation*}
The limit is independent of the chosen subsequence,  hence the whole family converges. 
We complete the proof of this lemma. 
\end{proof}

\begin{lemma}\label{08.05.L2}
	For each $k\in\mathbb N^*$, we have
	\begin{equation*}
		\pt_{x_2}\Phi_{\de,k}^{(2)}(1) \ra \pt_{x_2} \Phi_k^{(2)}(1) \mbox{ as } \de\ra 0^+, 
	\end{equation*}
	and 
	\begin{equation*}
	\left|\partial_{x_2}\Phi_{\delta,k}^{(2)}(1)\right|^2
	\le C\lambda_{\delta,k}^{(2)} \mbox{ for all } \de\in [0,4^{-1}], k\in\N^*, 
	\end{equation*}
	where $C>0$ depends only on $\alpha$.
\end{lemma}

\begin{proof}
	Since 
	\begin{equation*}
		-x_2^\al\pt_{x_2x_2}\Phi_{\de,k}^{(2)}=\la_{\de,k}^{(2)}\Phi_{\de, k}^{(2)}+\al x_{2}^{\al-1}\pt_{x_2} \Phi_{\de,k}^{(2)} \mbox{ on } (2^{-1},1), 
	\end{equation*}
	from Lemma \ref{08.05.L1}, we obtain 
	\begin{equation*}
		\pt_{x_2x_2}\Phi_{\de,k}^{(2)} \ra \pt_{x_2x_2}\Phi_k^{(2)} \mbox{ strongly in }L^2(2^{-1},1) \mbox{ as } \de\ra 0^+. 
	\end{equation*}
	Hence $\Phi_{\de,k}^{(2)}\ra \Phi_k^{(2)}$ strongly in $H^2(\f{1}{2},1)$ as $\de\ra 0^+$. By the embedding 
	$H^2(\f{1}{2},1)\hookrightarrow C^1[\f{1}{2},1]$, we have proved $\pt_{x_2}\Phi_{\de,k}^{(2)}(1)$ is meaningful, and 
	\begin{equation*}
		\pt_{x_2}\Phi_{\de,k}^{(2)}(1) \ra \pt_{x_2} \Phi_k^{(2)}(1) \mbox{ as } \de\ra 0^+. 
	\end{equation*}
	
	Let $\de\in (0,\f{1}{4}]$. Let $\zeta\in C_0^\iy(\R), 0\leq \zeta \leq 1$ such that 
	\begin{equation*}
		\zeta=1 \mbox{ on } \left(\f{3}{4},+\iy\right), \quad \zeta=0 \mbox{ on } \left(-\iy, \f{1}{2}\right), \quad |\zeta'|\leq 8 \mbox{ on }\R. 
	\end{equation*}
	Multiplying
	\begin{equation*}
		\begin{cases}
			\mcA_\de^{(2)}\Phi_{\de,k}^{(2)}=\la_{\de,k}^{(2)}\Phi_{\de,k}^{(2)}, &x\in \La_2^\de,\\
			\Phi_{\de,k}^{(2)}=0, &x\in \pt\La_2^\de
		\end{cases}
	\end{equation*}
	by $\zeta\pt_{x_2}\Phi_{\de,k}^{(2)}$, integrating on $(0,1)$, integration by parts, we obtain 
	\begin{equation*}
		\begin{split}
			&\f{1}{2}\left|\pt_{x_2}\Phi_{\de,k}^{(2)}(1)\right|^2 \\
			&=\int_{\La_2^\de} \zeta' x_2^\al (\pt_{x_2}\Phi_{\de,k}^{(2)})^2\df x_2-\f{1}{2}\int_{\La_2^\de} (\pt_{x_2}\Phi_{\de,k}^{(2)})^2 \pt_{x_2}(\zeta x_2^\al)\df x_2+\f{1}{2}\la_{\de,k}^{(2)}\int_{\La_2^\de}\zeta' (\Phi_{\de,k}^{(2)})^2\df x_2\\
			&\leq C\int_{\La_2^\de}x_2^\al (\pt_{x_2}\Phi_{\de,k}^{(2)})^2\df x_2+C\la_{\de,k}^{(2)}\int_{\La_2^\de}(\Phi_{\de,k}^{(2)})^2\df x_2\leq C\la_{\de,k}^{(2)}, 
		\end{split}
	\end{equation*}
	where the positive constants $C$ depends only on $\al$.  Passing to the limit as $\delta\to0^+$, 
	\begin{equation*} 
	|\partial_{x_2}\Phi_k^{(2)}(1)|^2
	\leq C\lambda_k^{(2)},
	\end{equation*} 
	where the positive constants $C$ depends only on $\al$. We prove the lemma. 
\end{proof}
 
\subsection{Approximation}

In this subsection, we treat the case $\de=0$ as the original problem \eqref{01.08.1} and the case $\de\in(0,\frac14)$ as the shape design problem \eqref{08.04.7}. Although the operators $\mathcal A_\delta$ have the same formal expression, the case $\delta=0$ is degenerate while the case $\delta>0$ is uniformly elliptic. Therefore, their domains and regularity properties are different.  Therefore, they must be treated separately in the spectral approximation. We also omit the symbol $E_\de^0$ (see \eqref{08.04.8} and \eqref{08.05.14} and \eqref{08.05.13})  whenever there is no risk of confusion.

Let $y^0, y^1\in C_0^\iy(\Om)$. Since the supports of $y^0$ and $y^1$ are separated from the degenerate boundary $\Ga_2^0=\{x_2=0\}$, for every
\begin{equation}\label{08.05.20}
	0\leq \delta<\de_0\equiv \f{1}{2}\min\left\{\operatorname{dist}(\operatorname{supp}y^0,\{0\}),
	\operatorname{dist}(\operatorname{supp}y^1,\{0\})\right\},
\end{equation}
we have
$y^0,y^1\in D(\mathcal A_\delta^\theta)$
for all $\theta\geq0$. Hence, when we take $y_\de^0=y^0$ and $y_\de^1=y^1$ for all $\de\in [0,\de_0)$, we have  
\begin{equation}\label{08.05.19}
	y^0=\sum_{n,k=1}^\iy y_{\de,nk}^0\Phi_{n}^{(1)}\Phi_{\de,k}^{(2)}, \quad y^1=\sum_{n,k=1}^\iy y_{\de,nk}^1\Phi_{n}^{(1)}\Phi_{\de,k}^{(2)}, 
\end{equation}
where 
\begin{equation}\label{08.05.18}
	y_{\de,nk}^0=(y^0, \Phi_{n}^{(1)}\Phi_{\de, k}^{(2)})_{L^2(\Om_\de)},\quad  y_{\de,nk}^1=(y^1, \Phi_{n}^{(1)}\Phi_{\de, k}^{(2)})_{L^2(\Om_\de)}. 
\end{equation}
Let  $M=\sup_{0\leq \de\leq \de_0}M_\de+1$ and 
\begin{equation*}
	\begin{split} 
		M_\de&=\sum_{n,k=1}^\iy \left[(\la_{nk}^\de)^2(y_{\de,nk}^0)^2+\la_{nk}^\de (y_{\de,nk}^1)^2\right]=\|\mcA_\de  y^0\|_{L^2(\Om_\de)}^2+\|y^1\|_{H_0^1(\Om_\de;w_\de)}^2\\
		&=\int_{\Om_\de}|\Div(A\nabla y^0)|^2\df x+\int_{\Om_\de} \nabla y^1\cdot A\nabla y^1\df x\\
		&= \int_{\Om_\de} \left(\pt_{x_1x_1}y^0+x_2^\al \pt_{x_2x_2}y^0+\al x_2^{\al-1}\pt_{x_2}y^0\right)^2\df x+\int_{\Om_\de} \left((\pt_{x_1}y^1)^2+x_2^\al (\pt_{x_2}y^1)^2\right)\df x\\
		&\leq C\left( \|D^2 y^0\|_{L^\iy(\Om)}^2+\|Dy^1\|_{L^\iy(\Om)}^2\right),  
	\end{split} 
\end{equation*}
here the positive constant $C$ depends only on $\al$ and $\de_0$. Then 
\begin{equation}\label{08.07.5}
	M\leq C<+\iy, 
\end{equation}
where the positive constant $C$ depends only on $\al$ and $\de_0$.

Now,  we introduce some shape design notation. 
Let $y_\de\in L^2(0,T; H_0^1(\Om_\de;w_\de))$. Define 
\begin{equation}\label{08.05.14}
	E_\de^0 y_\de(x,t)=
	\begin{cases}
		y_\de(x,t), &(x,t)\in \Om_\de \ts (0,T), \\
		0, &(x,t)\in (\Om-\Om_\de)\ts (0,T). 
	\end{cases}
\end{equation}
Then 
\begin{equation*}
	\|E_\de^0y_\de\|_{L^2(0,T; H_0^1(\Om;w))}=\|y_\de\|_{L^2(0,T; H_0^1(\Om_\de;w_\de))}. 
\end{equation*}

Let $y_\de\in H^1(0,T; L^2(\Om_\de))$. Define 
\begin{equation}\label{08.05.13}
	E_\de^0y_\de(x,t)=
	\begin{cases}
		y_\de(x,t), &(x,t)\in \Om_\de\ts (0,T),\\
		0, &(x,t)\in (\Om-\Om_\de)\ts (0,T), 
	\end{cases}
\end{equation}
then 
\begin{equation*}
	\pt_tE_\de^0 y_\de=E_\de^0\pt_ty_\de \mbox{ a.e.~on }Q. 
\end{equation*}
Indeed, for each $\zeta=\zeta(t)\in C_0^\iy(0,T)$ and each $v\in C_0^\iy(\Om)$, we have 
\begin{equation*}
	\begin{split}
		\iint_{Q} (\pt_tE_\de^0y_\de)\zeta v\df x_2\df t
		&=-\iint_{Q} (E_\de^0y_\de)\zeta' v\df x_2\df t=-\iint_{Q_\de}y_\de \zeta'v\df x_2\df t\\
		&=\iint_{Q_\de}(\pt_ty_\de)\zeta v\df x_2\df t=\iint_{Q}\left[E_\de^0(\pt_ty_\de)\right] \zeta v\df x_2\df t, 
	\end{split}
\end{equation*}
hence $\pt_tE_\de^0 y_\de=E_\de^0\pt_ty_\de$ a.e.~on $Q$. 

Now, we prove Theorem  \ref{08.05.T1}. We recall that 
\begin{equation*}
	\la_{nk}=\la_{nk}^0=\la_n^{(1)}+\la_{k}^{(2)},\quad \la_{nk}^\de=\la_n^{(1)}+\la_{\de, k}^{(2)}. 
\end{equation*}
For simplicity the notation, we denote 
\begin{equation}\label{08.07.2}
	\begin{split}
		A_{\de, nk}
		&=y_{\de, nk}^0\cos \left(\sqrt{\la_{nk}^\de}t\right)+y_{\de,nk}^1\f{1}{\sqrt{\la_{nk}^\de}}\sin \left(\sqrt{\la_{nk}^\de}t\right) \mbox{ for } \de\in [0, \de_0], \ A_{nk}=A_{0,nk}. 
	\end{split}
\end{equation}

\begin{theorem}\label{08.05.T1}
	Let $y^0,y^1\in C_0^\infty(\Omega)$. Let $y$ be the weak solution of \eqref{01.08.1} with initial data $(y^0,y^1)$, and $y_\de\ (0<\de\leq \f{1}{4})$ be the weak solution of \eqref{08.04.7} with initial data $(y^0,y^1)$. 
	Then
	\begin{equation}\label{08.05.15}
	E_\de^0y_\delta\to y \text{ strongly in }
	C([0,T];H_0^1(\Omega;w))
	\end{equation}
	and
	\begin{equation}\label{08.05.16} 
	\partial_tE_\de^0y_\delta\to\partial_ty \text{ strongly in }
	C([0,T];L^2(\Omega)), 
	\end{equation}
	and 
	\begin{equation}\label{08.05.17}
		\f{\pt E_\de^0y_\de}{\pt \nu_A}\ra \f{\pt y}{\pt \nu_A} \mbox{ strongly in } C([0, T]; L^2(\pt\Om-\Ga_N^0)). 
	\end{equation}
\end{theorem}

\begin{proof}
	We prove this by the following steps. 
	
	{\it Step 1}. 
	It is known that $y$ has the form \eqref{08.05.11}, and $y_\de\ (0<\de\leq \f{1}{4})$ has the form \eqref{08.05.10}. 
	
	{\it Step 2}. We compute
	\begin{equation*}
		\begin{split}
			\|y\|_{C([0,T]; H_0^1(\Om;w))}^2
			&\leq \sup_{t\in [0,T]}\sum_{n,k=1}^\iy \left(y_{nk}^0\cos\left(\sqrt{\la_{nk}}t\right)+y_{nk}^1 \f{1}{\sqrt{\la_{nk}}}\sin \left(\sqrt{\la_{nk}}t\right)\right)^2\la_{nk}\\
			&\leq 2\sum_{n,k=1}^\iy \left[\la_{nk}(y_{nk}^0)^2+(y_{nk}^1)^2\right]<+\iy
		\end{split}
	\end{equation*}
	and 
	\begin{equation*}
		\begin{split}
			\|\pt_ty\|_{C([0,T]; L^2(\Om))}^2
			&\leq \sup_{t\in [0,T]}\sum_{n,k=1}^\iy \left(-y_{nk}^0\sqrt{\la_{nk}}\sin \left(\sqrt{\la_{nk}}t\right) +y_{nk}^1\cos\left(\sqrt{\la_{nk}}t\right)\right)^2\\
			&\leq 2\sum_{n,k=1}^\iy \left[\la_{nk}(y_{nk}^0)^2+(y_{nk}^1)^2\right]<+\iy
		\end{split}
	\end{equation*}
	by \eqref{08.07.5}, and 
	\begin{equation*}
		\begin{split}
			\|y_\de\|_{C([0,T];  H_0^1(\Om_\de;w_\de))}^2
			&\leq \sup_{t\in [0,T]}\sum_{n,k=1}^\iy \left(y_{\de,nk}^0\cos\left(\sqrt{\la_{nk}^\de}t\right)+y_{\de,nk}^1\f{1}{\sqrt{\la_{nk}^\de}}\sin \left(\sqrt{\la_{nk}^\de}t\right)\right)^2\la_{nk}^\de\\
			&\leq 2\sum_{n,k=1}^\iy \left[\la_{nk}^\de(y_{\de, nk}^0)^2+(y_{\de, nk}^1)^2\right]<+\iy
		\end{split}
	\end{equation*}
	and 
	\begin{equation*}
		\begin{split}
			\|\pt_ty_\de\|_{C([0,T];L^2(\Om_\de))}^2 &\leq \sup_{t\in [0,T]}\sum_{n,k=1}^\iy \left(-y_{\de, nk}^0\sqrt{\la_{nk}^\de}\sin \left(\sqrt{\la_{nk}^\de}t\right) +y_{\de,nk}^1\cos\left(\sqrt{\la_{nk}^\de}t\right)\right)^2\\
			&\leq 2\sum_{n,k=1}^\iy \left[\la_{nk}^\de(y_{\de,nk}^0)^2+(y_{\de,nk}^1)^2\right]<+\iy
		\end{split}
	\end{equation*}
	by \eqref{08.07.5}. 
	
	{\it Step 3}. On $\Ga_1^{-1}$, by \eqref{08.05.11} and Remark \ref{08.05.R1},
	\begin{equation*}
		\partial_{x_1}\Phi_n^{(1)}(-1)=\sqrt{\lambda_n^{(1)}},
	\end{equation*}
	we obtain
	\begin{equation*}
		\frac{\partial y}{\partial\nu_A}
		=-\partial_{x_1}y
		=-\sum_{k=1}^\infty
		\left(
		\sum_{n=1}^\infty
		A_{nk}(t)\sqrt{\lambda_n^{(1)}}
		\right)\Phi_k^{(2)}.
	\end{equation*}
	Hence, by the orthonormality of $\{\Phi_k^{(2)}\}_{k=1}^\infty$ in $L^2(\La_2)$,
	\begin{equation*}
		\begin{aligned}
			\left\|\frac{\partial y}{\partial\nu_A}\right\|_{L^2(\Ga_1^{-1}\times(0,T))}^2
			&=\int_0^T\sum_{k=1}^\infty
			\left(
			\sum_{n=1}^\infty
			A_{nk}(t)\sqrt{\lambda_n^{(1)}}
			\right)^2\df t\\
			&=\int_0^T\sum_{k=1}^\infty
			\left(
			\sum_{n=1}^\infty
			(\lambda_n^{(1)})^{-\theta}
			A_{nk}(t)
			(\lambda_n^{(1)})^{\frac12+\theta}
			\right)^2\df t\\
			&\leq
			\left(\sum_{n=1}^\infty(\lambda_n^{(1)})^{-2\theta}\right)
			\int_0^T\sum_{n,k=1}^\infty
			(A_{nk}(t))^2(\lambda_n^{(1)})^{1+2\theta}\df t.
		\end{aligned}
	\end{equation*}
	Since
	\begin{equation*}
		(A_{nk}(t))^2
		\leq
		2\left[
		(y_{nk}^0)^2+
		(\lambda_{nk})^{-1}(y_{nk}^1)^2
		\right],
	\end{equation*}
	we have
	\begin{equation*}
		\begin{aligned}
			\left\|\frac{\partial y}{\partial\nu_A}\right\|_{L^2(\Ga_1^{-1}\times(0,T))}^2
			&\leq
			2T\left(\sum_{n=1}^\infty(\lambda_n^{(1)})^{-2\theta}\right)
			\sum_{n,k=1}^\infty
			\left[
			(y_{nk}^0)^2+
			(\lambda_{nk})^{-1}(y_{nk}^1)^2
			\right]
			(\lambda_n^{(1)})^{1+2\theta}\\
			&\leq
			2T\left(\sum_{n=1}^\infty(\lambda_n^{(1)})^{-2\theta}\right)
			\sum_{n,k=1}^\infty
			\left[
			(y_{nk}^0)^2+
			(\lambda_{nk})^{-1}(y_{nk}^1)^2
			\right]
			(\lambda_{nk})^{1+2\theta}.
		\end{aligned}
	\end{equation*}
	For $\theta\in(\frac14,\frac12]$, we have (Remark \ref{08.05.R1})
	\begin{equation*}
		\sum_{n=1}^\infty(\lambda_n^{(1)})^{-2\theta}<+\infty,
	\end{equation*}
	and
	\begin{equation*}
		(\lambda_{nk})^{1+2\theta}(y_{nk}^0)^2
		\leq
		(\lambda_{nk})^2(y_{nk}^0)^2,
		\qquad
		(\lambda_{nk})^{2\theta}(y_{nk}^1)^2
		\leq
		\lambda_{nk}(y_{nk}^1)^2.
	\end{equation*}
	Consequently,
	\begin{equation*}
		\left\|\frac{\partial y}{\partial\nu_A}\right\|_{L^2(\Ga_1^{-1}\times(0,T))}^2
		\leq
		CTM,
	\end{equation*}
	where $C>0$ is independent of $\delta$.
	
	On $\Ga_1^{1}$, arguing exactly as above and using
	\begin{equation*} 
	\partial_{x_1}\Phi_n^{(1)}(1)=(-1)^n\sqrt{\lambda_n^{(1)}},
	\end{equation*}
	we obtain
	\begin{equation*}
		\left\|\frac{\partial y}{\partial\nu_A}\right\|_{L^2(\Ga_1^{1}\times(0,T))}^2
		\leq CTM.
	\end{equation*}
	
	On $\Ga_2^1$, we have
	\begin{equation*} 
	\frac{\partial y}{\partial\nu_A}
	=x_2^\alpha\partial_{x_2}y
	=\partial_{x_2}y,
	\end{equation*} 
	since $x_2=1$ on $\Ga_2^1$. By the expression of $\Phi_k^{(2)}$ in Remark \ref{08.05.R1} and the properties of Bessel functions, we have
	\begin{equation*} 
	\left|x_2^\alpha\partial_{x_2}\Phi_k^{(2)}(1)\right|^2
	=2\kappa\lambda_k^{(2)},
	\end{equation*}
	see also Lemma \ref{08.05.L2}. Hence, by the orthonormality of $\{\Phi_n^{(1)}\}_{n=1}^\infty$ in $L^2(\La_1)$,
	\begin{equation*}
		\begin{aligned}
			\left\|\frac{\partial y}{\partial\nu_A}\right\|_{L^2(\Ga_2^1\times(0,T))}^2
			&=2\kappa\int_0^T
			\sum_{n=1}^\infty
			\left(
			\sum_{k=1}^\infty
			A_{nk}(t)\sqrt{\lambda_k^{(2)}}
			\right)^2\df t\\
			&=2\kappa\int_0^T
			\sum_{n=1}^\infty
			\left(
			\sum_{k=1}^\infty
			(\lambda_k^{(2)})^{-\theta}
			A_{nk}(t)
			(\lambda_k^{(2)})^{\frac12+\theta}
			\right)^2\df t\\
			&\leq
			2\kappa
			\left(\sum_{k=1}^\infty
			(\lambda_k^{(2)})^{-2\theta}\right)
			\int_0^T
			\sum_{n,k=1}^\infty
			(A_{nk}(t))^2
			(\lambda_k^{(2)})^{1+2\theta}\df t.
		\end{aligned}
	\end{equation*}
	Since
	\begin{equation*} 
	(A_{nk}(t))^2
	\leq
	2\left[
	(y_{nk}^0)^2+
	(\lambda_{nk})^{-1}(y_{nk}^1)^2
	\right],
	\end{equation*} 
	we obtain
	\begin{equation*}
		\begin{aligned}
			\left\|\frac{\partial y}{\partial\nu_A}\right\|_{L^2(\Ga_2^1\times(0,T))}^2
			&\leq
			CT
			\left(\sum_{k=1}^\infty
			(\lambda_k^{(2)})^{-2\theta}\right)
			\sum_{n,k=1}^\infty
			\left[
			(y_{nk}^0)^2+
			(\lambda_{nk})^{-1}(y_{nk}^1)^2
			\right]
			(\lambda_k^{(2)})^{1+2\theta}\\
			&\leq
			CT
			\left(\sum_{k=1}^\infty
			(\lambda_k^{(2)})^{-2\theta}\right)
			\sum_{n,k=1}^\infty
			\left[
			(y_{nk}^0)^2+
			(\lambda_{nk})^{-1}(y_{nk}^1)^2
			\right]
			(\lambda_{nk})^{1+2\theta}\leq CTM,
		\end{aligned}
	\end{equation*}
	for $\theta\in(\frac14,\frac12]$, where $C>0$ is independent of $\delta$.
	
	{\it Step 4}.  On $\Ga_{\de,1}^{\pm1}$, we have $
	\frac{\partial y_\de}{\partial\nu_{A_\de}}
	=\pm\partial_{x_1}y_\de$. 
	Arguing as in Step 3, for $\theta\in(\frac14,\frac12]$, we obtain
	\begin{equation*}
		\begin{aligned}
			\left\|\frac{\partial y_\de}{\partial\nu_{A_\de}}\right\|_{L^2(\Ga_{\de,1}^{\pm1}\times(0,T))}^2
			&\leq
			2T\left(\sum_{n=1}^\infty(\lambda_n^{(1)})^{-2\theta}\right)
			\sum_{n,k=1}^\infty
			\left[
			(y_{nk}^0)^2+
			(\lambda_{nk}^\de)^{-1}(y_{nk}^1)^2
			\right]
			(\lambda_{nk}^\de)^{1+2\theta}\\
			&\leq
			2T\left(\sum_{n=1}^\infty(\lambda_n^{(1)})^{-2\theta}\right)
			M_\de\leq CTM,
		\end{aligned}
	\end{equation*}
	where $C>0$ is independent of $\de$.
	
	On $\Ga_{\de,2}^{1}$, we have
	$
	\frac{\partial y_\de}{\partial\nu_{A_\de}}
	=\partial_{x_2}y_\de$. 
	By Lemmas \ref{08.04.L10} and  \ref{08.05.L2}, the orthonormality of ${\Phi_n^{(1)}}$, and the same weighted Cauchy–Schwarz argument as above, for $\theta\in(\frac14,\frac12]$,
	\begin{equation*}
		\begin{split} 
			\left\|\frac{\partial y_\de}{\partial\nu_{A_\de}}\right\|_{L^2(\Ga_{\de,2}^{1}\times(0,T))}^2
			\leq
			CT\left(\sum_{k=1}^\infty(\lambda_{\de,k}^{(2)})^{-2\theta}\right)
			\sum_{n,k=1}^\infty
			\left[(y_{nk}^0)^2+(\lambda_{nk}^\de)^{-1}(y_{nk}^1)^2\right]
			(\lambda_{nk}^\de)^{1+2\theta}
			\leq CTM,
		\end{split} 
	\end{equation*} 
	where $C>0$ is independent of $\de$.
	
	{\it Step 5}.  
	From \eqref{08.05.18} and Lemma \ref{08.05.L1}, we obtain 
	\begin{equation*}
		\begin{split}
			y_{\de,nk}^0=(y^0,\Phi_n^{(1)}\Phi_{\de,k}^{(2)})_{L^2(\Om)}\ra (y^0, \Phi_n^{(1)}\Phi_k^{(2)})_{L^2(\Om)}=y_{nk}^0 \mbox{ as } \de \ra 0^+, 
		\end{split}
	\end{equation*}
	and 
	\begin{equation*}
		\begin{split}
			y_{\de,nk}^1=(y^1,\Phi_n^{(1)}\Phi_{\de,k}^{(2)})_{L^2(\Om)}\ra (y^1, \Phi_n^{(1)}\Phi_k^{(2)})_{L^2(\Om)}=y_{nk}^1 \mbox{ as } \de \ra 0^+.  
		\end{split}
	\end{equation*}
	
	{\it Step 6}. We prove \eqref{08.05.15}. 
	
	Let $\e>0$. Choose $m\in\N^*$ large enough such that for all $n,k\geq m+1$, we have 
	\begin{equation}\label{08.07.3}
		\la_n^{(1)}\geq \max\{1, M\e^{-2}\},\quad \la_k^{(2)}\geq \max\{1,M\e^{-2}\}, 
	\end{equation} 
	Then, for all $0<\de\leq \de_0$, we have  $\la_{\de,k}^{(2)}\geq \la_k^{(2)}\geq \max\{1,M\e^{-2}\}$ for all $k\geq m+1$ by \eqref{08.07.1}, and  $\la_{nk}^\de=\la_n^{(1)}+\la_{\de,k}^{(2)}\geq M\e^{-2}$ for all $n\geq m+1$ or $k\geq m+1$, and 
	\begin{equation}\label{08.06.1}
		\begin{split} 
		 \sum_{n,k\in A_m} \left(\la_{nk}^\de(y_{\de,nk}^0)^2+(y_{\de,nk}^1)^2\right)
		 &\leq \sum_{n,k\in A_m} (\la_{nk}^\de)^{-1} \la_{nk}^\de  \left(\la_{nk}^\de (y_{\de,nk}^0)^2+(y_{\de,nk}^1)^2\right)\\
		 &\leq \sum_{n,k\in A_m} (\la_{nk})^{-1} \la_{nk}^\de  \left(\la_{nk}^\de (y_{\de,nk}^0)^2+(y_{\de,nk}^1)^2\right)\\
		 &<M_\de\f{\e^2}{M}<\e^2.  
		 \end{split} 
	\end{equation} 
	For convenience, throughout the remainder of the proof we set
	\begin{equation*}
	A_m=\{(n,k)\in\N^2:\ n\ge m+1 \text{ or } k\ge m+1\}.
	\end{equation*}
	
	 From \eqref{08.05.11} and \eqref{08.05.10}, we get 
	 \begin{equation*}
	 	\begin{split}
	 		y_\de-y
	 		&=\sum_{n,k=1}^m \left(A_{\de,nk}\Phi_n^{(1)}\Phi_{\de,k}^{(2)}-A_{nk}\Phi_n^{(1)}\Phi_k^{(2)}\right)\\
	 		&\hspace{4.5mm}+\sum_{n,k\in A_m} A_{\de,nk}\Phi_n^{(1)}\Phi_{\de,k}^{(2)}+\sum_{n,k\in A_m} A_{nk}\Phi_n^{(1)}\Phi_k^{(2)}\\
	 		&\equiv A_1+A_2+A_3.
	 	\end{split} 
	 \end{equation*} 
	 
	 From \eqref{08.06.1}, together with  $|\cos s|\leq1, |\sin s|\leq1$,
	 we deduce that 
	 \begin{equation*}
	 	\begin{split} 
	 	\lambda_{nk}^\de |A_{\de,nk}|^2
	 	&=
	 	\lambda_{nk}^\de
	 	\left|
	 	y_{\de,nk}^0\cos\left(\sqrt{\la_{nk}^\de}t\right)+\frac{y_{\de,nk}^1}{\sqrt{\lambda_{nk}^\de}}\sin\left(\sqrt{\la_{nk}^\de}t\right)
	 	\right|^2 \le
	 	2\lambda_{nk}^\de (y_{\de,nk}^0)^2
	 	+
	 	2(y_{\de,nk}^1)^2.
	 	\end{split} 
	 \end{equation*}
	 Therefore, from \eqref{08.06.1}, we obtain 
	 \begin{equation*}
	 	\|A_2\|_{C([0,T]; H_0^1(\Om;w))}^2\leq 2
	 	\sum_{n,k\in A_m}
	 	\left[
	 	\lambda_{nk}^\delta
	 	(y_{\delta,nk}^0)^2
	 	+
	 	(y_{\delta,nk}^1)^2
	 	\right]<2\e, \mbox{ and } \|A_3\|_{C([0,T]; H_0^1(\Om;w))}<\sqrt{2}\e. 
	 \end{equation*}

	 Note that 
	 \begin{equation*}
	 	\begin{split}
	 		&\sum_{n,k=1}^m \left(A_{\de,nk}\Phi_n^{(1)}\Phi_{\de,k}^{(2)}-A_{nk}\Phi_n^{(1)}\Phi_k^{(2)}\right)=\sum_{n,k=1}^m\left[(A_{\de,nk}-A_{nk})\Phi_n^{(1)}\Phi_{\de,k}^{(2)}+ A_{nk}\Phi_n^{(1)}\left(\Phi_{\de,k}^{(2)}-\Phi_k^{(2)}\right)\right], 
	 	\end{split}
	 \end{equation*}
	 and 
	 \begin{equation*}
	 	\begin{split}
	 		&A_{\de,nk}-A_{nk}\\
	 		&=(y_{\de,nk}^0-y_{nk}^0)\cos\left(\sqrt{\la_{nk}^\de }t\right)+y_{nk}^0\left[\cos\left(\sqrt{\la_{nk}^\de }t\right)-\cos\left(\sqrt{\la_{nk}}t\right)\right]\\
	 		&\hspace{4.5mm}+(y_{\de,nk}^1-y_{nk}^1)\sin\left(\sqrt{\la_{nk}^\de }t\right)+y_{nk}^1\left[\sin\left(\sqrt{\la_{nk}^\de }t\right)-\sin\left(\sqrt{\la_{nk}}t\right)\right], 
	 	\end{split}
	 \end{equation*}
	 and (Step 5)
	 \begin{equation*}
	 	\begin{split}
	 		y_{\de,nk}^0\ra y_{nk}^0, \quad y_{\de,nk}^1\ra y_{nk}^1, 	 	\end{split}
	 \end{equation*}
	 and
	 \begin{equation*}
	 	\begin{split}
	 		\left|\cos\left(\sqrt{\la_{nk}^\de }t\right)-\cos\left(\sqrt{\la_{nk}}t\right)\right|
	 		&\leq \left|\sqrt{\la_{nk}^\de }-\sqrt{\la_{nk}}\right|t,\\
	 		\left|\sin\left(\sqrt{\la_{nk}^\de }t\right)-\sin\left(\sqrt{\la_{nk}}t\right)\right|
	 		&\leq \left|\sqrt{\la_{nk}^\de }-\sqrt{\la_{nk}}\right|t, 
	 	\end{split}
	 \end{equation*}
	 by Lemmas \ref{08.04.L10} and \ref{08.05.L1}, we have
	 \begin{equation}\label{08.07.4} 
	 \lambda_{nk}^\de\rightarrow\lambda_{nk},\qquad
	 \Phi_{\de,k}^{(2)}\rightarrow\Phi_k^{(2)} \mbox{ strongly in } H_0^1(\La_2;w)
	 \end{equation} 
	 as $\de\rightarrow0^+$, hence, since the above summation contains only finitely many terms, we conclude that
	 \begin{equation*} 
	 \|A_1\|_{C([0,T];H_0^1(\Om;w))}\rightarrow0.
	 \end{equation*} 
	 Combining the above convergence with the estimates for $A_2$ and $A_3$, we conclude that, for all sufficiently small $\de>0$,
	 \begin{equation*} 
	 \|y_\de-y\|_{C([0,T];H_0^1(\Om;w))}
	 \leq C\e ,
	 \end{equation*} 
	 where $C>0$ is independent of $\de\in[0,\de_0]$. Since $\epsilon>0$ is arbitrary, \eqref{08.05.15} follows.
	 
	 {\it Step 7}. The proof is analogous. Indeed, differentiating the spectral representation with respect to time only exchanges the roles of the coefficients $y_{\delta,nk}^{0}$ and $y_{\delta,nk}^{1}$ in the energy estimate.
	 
	 {\it Step 8}. We prove \eqref{08.05.17}. 
	 
	 On $\Ga_{1}^{-1}$, we have $\f{\pt y_\de}{\pt \nu_A}=- \pt_{x_1}y_\de$ for all $\de \in [0,\de_0]$, and 
	 \begin{equation*}
	 	\begin{split}
	 		\f{\pt y_\de}{\pt \nu_A}-\f{\pt y}{\pt \nu_A}
	 		&=-\sum_{n,k=1}^m \left(A_{\de,nk}\Phi_{\de,k}^{(2)}-A_{nk}\Phi_{k}^{(2)}\right)(\pt_{x_1}\Phi_n^{(1)}(-1))\\
	 		&\hspace{4.5mm}- \sum_{n,k\in A_m}A_{\de,nk}\Phi_{\de,k}^{(2)}(\pt_{x_1}\Phi_n^{(1)}(-1))- \sum_{n,k\in A_m}A_{nk}\Phi_{k}^{(2)}(\pt_{x_1}\Phi_n^{(1)}(-1))\\
	 		&=B_1+B_2+B_3
	 	\end{split}
	 \end{equation*}
	 by \eqref{08.05.11} and \eqref{08.05.10}, where $A_{\de,nk}\ (\de\in [0,\de_0])$ is defined in \eqref{08.07.2}, and $m$ is defined in \eqref{08.07.3}. From Steps 3 and  4, we get 
	 \begin{equation*}
	 	\|B_2\|_{L^2(\Ga_1^{-1}\ts (0,T))}\leq \sqrt{2}\e, \quad \|B_3\|_{L^2(\Ga_1^{-1}\ts (0,T))}\leq \sqrt{2}\e. 
	 \end{equation*}
	 Arguing exactly as in Step 6, from \eqref{08.07.4} and Remark \ref{08.05.R1}, we deduce that 
	 \begin{equation*}
	 	\begin{split}
	 		\|B_1\|_{L^2(\Ga_1^{-1}\ts (0,T))}\leq \e
	 	\end{split}
	 \end{equation*}
	 for $\de>0$ small enough. Hence we have 
	 \begin{equation*}
	 	\f{\pt y_\de}{\pt \nu_A}\ra \f{\pt y}{\pt \nu_A} \mbox{ strongly in }L^2(\Ga_1^{-1}\ts (0,T)) \mbox{ as } \de\ra 0^+. 
	 \end{equation*} 
	 The convergence on the remaining boundary components follows in the same way by applying Remark \ref{08.05.R1} and  Lemma \ref{08.05.L2}.
	 We complete the proof of this theorem. 
\end{proof}

\begin{remark}\label{08.07.R1}
	From the proof of Theorem \ref{08.05.T1}, the same conclusion remains valid if, for some $\theta>\frac14$,
	\begin{equation*} 
	y^0\in D\left(\mathcal A^{\frac12+\theta}\right),
	\qquad
	y^1\in D\left(\mathcal A^\theta\right).
	\end{equation*}
	Here, the condition $\theta>\f{1}{4}$ arises from the summability property
	\begin{equation*} 
	\sum_{n=1}^\infty(\lambda_n^{(1)})^{-2\theta}<+\infty.
	\end{equation*} 
\end{remark}

\begin{lemma}\label{08.06.L1}
	Let $y^0,y^1\in C_0^\infty(\Omega)$, and let
	\begin{equation*} 
		0<\de<\de_0\equiv 
		\f{1}{2}\min\left\{
		\f{1}{4}, 
		\operatorname{dist}(\operatorname{supp}y^0,{0}),
		\operatorname{dist}(\operatorname{supp}y^1,{0})
		\right\}.
	\end{equation*} 
	Let $y_\de$ be the weak solution of \eqref{08.04.7} with initial data $(y^0,y^1)$. Then
	\begin{equation*} 
		y_\de
		\in
		L^2(0,T;H^2(\Omega_\de))
		\cap
		H^1(0,T;H_0^1(\Omega_\de))
		\cap
		H^2(0,T;L^2(\Omega_\de)).
	\end{equation*}
\end{lemma}

\begin{proof}
	We prove this corollary by the following steps. 
	
	{\it Step 1}. Let $\de\in (0,\de_0)$. Arguing exactly as in Step 6 of the proof of Theorem \ref{08.05.T1}, and using the characterization of $D(\mathcal A_\delta^\f{3}{2})$, we obtain 
	\begin{equation*}
		\begin{split} 
			&\sum_{n,k=1}^\iy(\la_{nk}^\de)^2 \left[(\la_{nk}^\de)(y_{\de,nk}^0)^2+(y_{\de,nk}^1)^2\right]=\|\mcA_\de ^\f{3}{2} y^0\|_{L^2(\Om_\de)}^2+\|\mcA_\de y^1\|_{L^2(\Om_\de)}^2\\ 
			&\leq C\left( \|D^3 y^0\|_{L^\iy(\Om)}^2+\|D^2y^1\|_{L^\iy(\Om)}^2\right),  
		\end{split} 
	\end{equation*}
	where the positive constant $C$ depends only on $\alpha$ and $\delta_0$, and is independent of $\delta$. 
	
	{\it Step 2}. We prove $y_\de\in H^2(\Om_\de)$. 
	
	Since the weak solution $y_\de$ of \eqref{08.04.7}  has the form  \eqref{08.05.10}, for $a+b=2$ with $a,b\in\N$,  and 
	\begin{equation*}
		\begin{split}
			\pt_{x_1}^a\pt_{x_2}^b y_\de^m
			&\equiv \pt_{x_1}^a\pt_{x_2}^b \sum_{n,k=1}^m \left(y_{\de,nk}^0\cos\left(\sqrt{\la_{nk}^\de}t\right)+y_{\de,nk}^1\f{1}{\sqrt{\la_{nk}^{\de}}}\sin \left(\sqrt{\la_{nk}^\de}t\right)\right)\Phi_n^{(1)}\Phi_{\de,k}^{(2)}\\
			&=\sum_{n,k=1}^m \left(y_{\de,nk}^0\cos\left(\sqrt{\la_{nk}^\de}t\right)+y_{\de,nk}^1\f{1}{\sqrt{\la_{nk}^{\de}}}\sin \left(\sqrt{\la_{nk}^\de}t\right)\right)\left([\pt_{x_1}^a\Phi_n^{(1)}]\pt_{x_2}^b\Phi_{\de,k}^{(2)}\right)
		\end{split}
	\end{equation*}
	we obtain for $m\in\N^*$ large enough, 
	\begin{equation*}
		\begin{split}
			&\left\|\pt_{x_1}^a\pt_{x_2}^by_\de^m\right\|_{L^2(\Om_\de)}^2\\
			&=\left\| \sum_{n,k=1}^m \left(y_{\de,nk}^0\cos\left(\sqrt{\la_{nk}^\de}t\right)+y_{\de,nk}^1\f{1}{\sqrt{\la_{nk}^{\de}}}\sin \left(\sqrt{\la_{nk}^\de}t\right)\right)\left([\pt_{x_1}^a\Phi_n^{(1)}]\pt_{x_2}^b\Phi_{\de,k}^{(2)}\right)\right\|_{L^2(\Om_\de)}^2\\ 
			&\leq C\sum_{n,k=1}^m \left((y_{\de,nk}^0)^2+(y_{\de,nk}^1)^2\f{1}{\la_{nk}^\de}\right)(\la_n^{(1)})^a (\la_{\de,k}^{(2)})^b\leq C\sum_{n,k=1}^m \left((y_{\de,nk}^0)^2+(y_{\de,nk}^1)^2\f{1}{\la_{nk}^\de}\right)(\la_{nk}^\de)^2\\
			&\leq C\sum_{n,k=1}^\iy(\la_{nk}^\de)^2 \left[(\la_{nk}^\de)(y_{\de,nk}^0)^2+(y_{\de,nk}^1)^2\right]
		\end{split}
	\end{equation*}
	by 
	\begin{equation*}
		\begin{split}
			\left|\pt_{x_2}\Phi_{\de,k}^{(2)}\right|=x_2^{-\f{\al}{2}}\left|x_2^\f{\al}{2}\pt_{x_2}\Phi_{\de,k}^{(2)}\right|\leq \de^{-\f{\al}{2}}\left|x_2^\f{\al}{2}\pt_{x_2}\Phi_{\de,k}^{(2)}\right|
		\end{split}
	\end{equation*}
	and 
	\begin{equation*}
		\begin{split}
			\left|\pt_{x_2}^2\Phi_{\de,k}^{(2)}\right|
			&\leq x_2^{-\al}\left|x_2^\al \pt_{x_2}^2\Phi_{\de,k}^{(2)}+\al x_2^{\al-1} \pt_{x_2}\Phi_{\de,k}^{(2)}\right|+\al x_2^{-\f{\al}{2}-1}\left|x_2^\f{\al}{2}\pt_{x_2}\Phi_{\de,k}^{(2)}\right|\\
			&\leq \de^{-\al}\left|\mcA_\de^{(2)}\Phi_{\de,k}^{(2)}\right| +\al\de^{-\f{\al}{2}-1} \left|x_2^\f{\al}{2}\pt_{x_2}\Phi_{\de,k}^{(2)}\right|,  
		\end{split}
	\end{equation*}
	where the positive constant $C$ depends only on $\al$ and $\de$.  
	Combining this estimate with Step 1 yields
	\begin{equation*}
		\sup_{m\ge1}
		\left\|
		\partial_{x_1}^a\partial_{x_2}^b
		y_\delta^m
		\right\|_{L^2(\Omega_\delta)}
		<+\infty .
	\end{equation*}
	Since $y_\delta^m\to y_\delta$ in
	$L^2(\Omega_\delta)$, the uniform estimate above implies, by weak compactness,
	$ \partial_{x_1}^a\partial_{x_2}^by_\delta \in  L^2(\Omega_\delta)
	$ 
	for every $a+b=2$, $a,b\in\mathbb N$. Hence,
	\begin{equation*}
		y_\delta\in H^2(\Omega_\delta).
	\end{equation*}
	
	{\it Step 3}. We prove that
	\begin{equation*}
		y_\delta\in H^1(0,T;H_0^1(\Omega_\delta))
		\cap
		H^2(0,T;L^2(\Omega_\delta)).
	\end{equation*}
	
	Differentiating the series representation \eqref{08.05.10} term by term with respect to $t$, and arguing exactly as in Step 2, we obtain uniform bounds for the partial sums of $\partial_t y_\delta$  in
	$L^2(0,T;H_0^1(\Omega_\delta))$
	and of
	$\partial_t^2y_\delta$ in
	$L^2(0,T;L^2(\Omega_\delta))$.
	Passing to the limit yields
	which implies that
	\begin{equation*}
		y_\delta\in H^1(0,T;H_0^1(\Omega_\delta))
		\cap
		H^2(0,T;L^2(\Omega_\delta)).
	\end{equation*}
	This completes the proof.
\end{proof}

Denote 
\begin{equation*}
	E(\de;t)=\f{1}{2}\int_{\Om_\de} \left((\pt_ty_\de)^2+\nabla y_\de\cdot A_\de\nabla y_\de\right)\df x\df t, \quad E(0;t)=E(t), 
\end{equation*}
then 
\begin{equation}\label{08.07.6}
	E(\de;t)=E(\de;0) \mbox{ for all }t\in [0,T] \mbox{ and } \de\in [0,\de_0]. 
\end{equation}

\begin{corollary}\label{08.07.C1}
	Under the assumptions of Lemma \ref{08.06.L1}, there exists a positive
	constant $C$, independent of $\delta\in(0,\delta_0)$, such that
	\begin{equation*}
		\left\|
		\frac{\partial y_\delta}{\partial\nu_{A_\delta}}
		\right\|_{L^2((\partial\Omega-\Gamma_{\delta,2}^{\delta})\times(0,T))}
		\leq
		CT
		\left(
		\|y^0\|_{H_0^1(\Omega;w)}
		+
		\|y^1\|_{L^2(\Omega)}
		\right).
	\end{equation*}
\end{corollary}

\begin{proof}
	Let $\zeta=\zeta(x_1)\in C^\iy(\R), 0\leq \zeta\leq 1$ such that 
	\begin{equation*}
		\zeta=1 \mbox{ on } (-\iy, 0), \quad \zeta=0 \mbox{ on } (2^{-1},+\iy), \quad |\pt_{x_1}\zeta|\leq 4 \mbox{ on } \R. 
	\end{equation*}
	Multiplying \eqref{08.04.7} by $\zeta \pt_{x_1}y_\de$, integrating on $Q_\de$, we obtain 
	\begin{equation*}
		0=\iint_{Q_\de} (\pt_{tt}y_\de) \zeta \pt_{x_1}y_\de\df x\df t-\iint_{Q_\de} [\Div(A_\de\nabla y_\de)] \zeta \pt_{x_1}y_\de\df x\df t\equiv A_1+A_2. 
	\end{equation*}
	
	Since $y_\de=0$ on $\Sigma_\de$, we obtain  $\pt_ty_\de=0$ on $\Sigma_\de$, and 
	\begin{equation*}
		\begin{split}
			A_1
			&=\int_{\Om_\de} (\pt_ty_\de) \zeta \pt_{x_1}y_\de \df x\bigg|_{t=0}^{t=T}+\f{1}{2}\iint_{Q_\de} (\pt_ty_\de)^2\pt_{x_1}\zeta \df x\df t. 
		\end{split}
	\end{equation*}
	
	From \eqref{08.05.10} and Lemma \ref{08.06.L1},   since $y_\delta=0$ on
	$\Gamma_{\delta,1}^{-1}\cup\Gamma_{\delta,1}^{1}$,
	we have
	\begin{equation*}
		\pt_{x_1}y_\de=0 \mbox{ on }\Ga_{\de,2}^\de\cup \Ga_{\de,2}^1, \quad \partial_{x_2}y_\delta=0 \text{ on }
		\Gamma_{\delta,1}^{-1}\cup\Gamma_{\delta,1}^{1}.
	\end{equation*}
	Then 
	\begin{equation*}
		\begin{split}
			A_2
			&=\f{1}{2}\iint_{\Ga_{\de,1}^{-1}\ts (0,T)}(\pt_{x_1}y_\de)^2\df S\df t\\
			&\hspace{4.5mm}+\iint_{Q_\de} (\pt_{x_1}\zeta)(\pt_{x_1}y_\de)^2\df x\df t-\f{1}{2}\iint_{Q_\de} (\pt_{x_1}\zeta)\nabla y_\de\cdot A_\de \nabla y_\de\df x\df t. 
		\end{split}
	\end{equation*}
	Hence
	\begin{equation*} 
		\begin{aligned}
			\frac12
			\iint_{\Gamma_{\delta,1}^{-1}\times(0,T)}
			|\partial_{x_1}y_\delta|^2\df S\df t
			&=
			-\left.
			\int_{\Omega_\delta}
			\partial_ty_\delta
			\zeta\partial_{x_1}y_\delta\df x
			\right|_0^T 
			-\frac12
			\iint_{Q_\delta}
			(\partial_{x_1}\zeta)
			|\partial_ty_\delta|^2\df x\df t
			\\
			&\quad
			-\iint_{Q_\delta}
			(\partial_{x_1}\zeta)
			|\partial_{x_1}y_\delta|^2\df x\df t 
			+\frac12
			\iint_{Q_\delta}
			(\partial_{x_1}\zeta)
			\nabla y_\delta\cdot A_\delta\nabla y_\delta\df x\df t .
		\end{aligned}
	\end{equation*}
	Hence, from \eqref{08.07.6} and $|\pt_{x_1}\zeta|\leq 4$, we get 
	\begin{equation*}
		\begin{split}
			\iint_{\Ga_{\de,1}^{-1}\ts (0,T)}(\pt_{x_1}y_\de)^2\df S\df t
			&\leq CTE(\de;0),  
		\end{split}
	\end{equation*}
	where the positive constant $C$ is absolute.  
	
	The same argument gives the corresponding estimates on
	$\Gamma_{\delta,1}^{1}$ and $\Gamma_{\delta,2}^{1}$.
	Therefore, the desired boundary estimate follows. 
	This completes the proof.
\end{proof}

\begin{corollary}\label{08.07.C2}
	Under the assumptions of Theorem \ref{08.05.T1}, there exists a
	constant $C>0$, independent of $\delta$ and $T$, such that
	\begin{equation*}
		\left\|
		\frac{\partial y}{\partial \nu_A}
		\right\|_{L^2((\partial\Omega-\Gamma_2^0)\times(0,T))}
		\leq
		CT
		\left(
		\|y^0\|_{H_0^1(\Omega;w)}
		+
		\|y^1\|_{L^2(\Omega)}
		\right).
	\end{equation*}
\end{corollary}

\begin{proof}
	By \eqref{08.05.17}, we have
	\begin{equation*}
		\frac{\partial y_\delta}{\partial\nu_{A_\delta}}
		\ra
		\frac{\partial y}{\partial\nu_A}
		\quad\text{strongly in }
		L^2((\partial\Omega-\Gamma_N^0)\times(0,T))
	\end{equation*}
	as $\delta\to0^+$. On the other hand, Corollary \ref{08.07.C1} gives the uniform estimate
	\begin{equation*}
		\left\|
		\frac{\partial y_\delta}{\partial\nu_{A_\delta}}
		\right\|_{L^2((\partial\Omega-\Gamma_{\delta,2}^{\delta})\times(0,T))}
		\leq
		CT
		\left(
		\|y^0\|_{H_0^1(\Omega;w)}
		+
		\|y^1\|_{L^2(\Omega)}
		\right),
	\end{equation*}
	where $C$ is independent of $\delta$ and $T$. Passing to the limit as $\delta\to0^+$ and using the above strong convergence, we obtain
	\begin{equation*}
		\left\|
		\frac{\partial y}{\partial\nu_A}
		\right\|_{L^2((\partial\Omega-\Gamma_2^0)\times(0,T))}
		\leq
		CT
		\left(
		\|y^0\|_{H_0^1(\Omega;w)}
		+
		\|y^1\|_{L^2(\Omega)}
		\right).
	\end{equation*}
	This completes the proof.
\end{proof}

\begin{remark}\label{08.06.R1}
	In Theorem \ref{08.05.T1} and Corollary \ref{08.07.C1}, no estimates are established for
	$\frac{\partial y_\delta}{\partial \nu_A}$ on $\Gamma_{\delta,2}^\delta$
	or
	$\frac{\partial y}{\partial \nu_A}$ on $\Gamma_2^0$,
	as they are not needed in the subsequent analysis. Nevertheless, the proof above shows immediately that
	\begin{equation*}
	\frac{\partial y_\delta}{\partial \nu_A}
	\in
	L^2(\Gamma_{\delta,2}^{\delta}\times(0,T)), \mbox{ for all } \de\in (0,\de_0].
	\end{equation*}
\end{remark}

\section{Observability inequality}\label{S4}

In this section, we establish an observability inequality for equation \eqref{01.08.1} by passing to the limit in the observability inequalities satisfied by the approximating uniformly hyperbolic equation \eqref{08.04.7}. 

Let $y^0,y^1\in C_0^\infty(\Omega)$, and 
\begin{equation*}
	\de_0= \f{1}{2}\min\left\{\operatorname{dist}(\operatorname{supp}y^0,\{0\}),
	\operatorname{dist}(\operatorname{supp}y^1,\{0\})\right\}.
\end{equation*} 
We first establish an observability estimate for the approximating problem \eqref{08.04.7}. Passing to the limit as $\delta\to0^+$ will then yield the observability inequality for the degenerate equation.

We choose the multiplier
\begin{equation*}
	H(x)=x\cdot\nabla y_\delta.
\end{equation*}
Multiplying both sides of \eqref{08.04.7} by $H$ and integrating over $Q_\delta=\Omega\times(0,T)$, we obtain
\begin{equation}\label{08.06.2}
	0=\iint_{Q_\delta} (\partial_{tt}y_\delta)H\,dx\,dt
	-\iint_{Q_\delta}\Div(A_\delta\nabla y_\delta)H\,dx\,dt
	\equiv A_1+A_2.
\end{equation}
In the following computations, all integrations by parts are justified by Lemma  \ref{08.06.L1}.  

Directly compute $\pt_ty_\de=0$ on $\pt \Om_\de\ts (0,T)$ by \eqref{08.05.10}, we obtain 
\begin{equation*}
	\begin{split}
		A_1
		&=\int_{\Om_\de} (\pt_ty_\de)(x\cdot \nabla y_\de)\df x\bigg|_{t=0}^{t=T}+\iint_{Q_\de} (\pt_ty_\de)^2\df x\df t. 
	\end{split}
\end{equation*}

Since 
\begin{equation*}
	\begin{split}
		A_2
		&=-\iint_{\pt Q_\de} (\nabla y_\de\cdot A_\de\nu)(x\cdot \nabla y_\de)\df S\df t+\f{1}{2}\iint_{\pt Q_\de} (\nabla y_\de\cdot A_\de\nabla y_\de)(x\cdot \nu)\df S\df t\\
		&\hspace{4.5mm}-\f{\al}{2}\iint_{Q_\de} x_2^\al (\pt_{x_2}y_\de)^2\df x\df t 
	\end{split}
\end{equation*}
by integration by parts, 
note that from
\begin{equation*}
	\begin{split}
		\nabla y_\de=( \pt_{x_1}y_\de, 0) \mbox{ on }\Ga_{\de,1}^{\pm1}, \quad \nabla y_\de=(0,  \pt_{x_2}y_\de) \mbox{ on }\Ga_{\de,2}^\de\cup \Ga_{\de,2}^1
	\end{split}
\end{equation*}
by $y_\delta=0$ on $\partial\Omega_\delta$, all tangential derivatives vanish. Hence, we obtain $\nabla y_\de\parallel \nu$ on $\pt Q_\de$, and 
\begin{equation*}
	\iint_{\pt Q_\de} (\nabla y_\de\cdot A_\de\nu)(x\cdot \nabla y_\de)\df S\df t =\iint_{\pt Q_\de}   (\nabla y_\de\cdot A_\de \nabla y_\de) (x\cdot \nu)\df S\df t, 
\end{equation*}
then 
\begin{equation*}
	\begin{split}
		A_2=-\f{1}{2}\iint_{\pt Q_\de} (\nabla y_\de\cdot A_\de\nabla y_\de) (x\cdot \nu)\df S\df t-\f{\al}{2}\iint_{Q_\de} x_2^\al (\pt_{x_2}y_\de)^2\df x\df t. 
	\end{split}
\end{equation*}

Combining $A_1$ and $A_2$, and $x\cdot \nu=-x_2<0$ on $\Ga_{\de,2}^\de$ (the corresponding boundary integral is nonnegative and may be discarded), we get
\begin{equation}\label{08.06.3}
	\begin{split}
		&\f{1}{2}\iint_{(\pt\Om_\de-\Ga_{\de,2}^\de)\ts (0,T)} (\nabla y_\de\cdot A_\de\nabla y_\de)(x\cdot \nu)\df S\df t\\
		&\geq \int_{\Om_\de} (\pt_t y_\de) (x\cdot \nabla y_\de)\df x\bigg|_{t=0}^{t=T}+\iint_{Q_\de} (\pt_ty_\de)^2\df x\df t-\f{\al}{2}\iint_{Q_\de} x_2^\al (\pt_{x_2}y_\de)^2\df x\df t. 
	\end{split}
\end{equation}

Rearranging the terms gives
\begin{equation}\label{08.06.4}
	\begin{split}
		&\iint_{Q_\de}(\pt_ty_\de)^2\df x\df t-\f{\al}{2}\iint_{Q_\de}x_2^\al (\pt_{x_2}y_\de)^2\df x\df t\\
		&=\f{2-\al}{4}\iint_{Q_\de}   \left((\pt_ty_\de)^2+\nabla y_\de\cdot A_\de\nabla y_\de\right)\df x\df t+\f{2+\al}{4}\iint_{Q_\de} (\pt_ty_\de)^2 \df x\df t\\
		&\hspace{4.5mm}-\f{2-\al}{4}\iint_{Q_\de} \nabla y_\de \cdot A_\de \nabla y_\de\df x\df t-\f{\al}{2}\iint_{Q_\de} x_2^\al (\pt_{x_2}y_\de)^2\df x\df t\\
		&\geq \f{2-\al}{2}TE(\de;0)+\f{2+\al}{4}\iint_{Q_\de} \left((\pt_ty_\de)^2-\nabla y_\de \cdot A_\de \nabla y_\de\right) \df x\df t
	\end{split}
\end{equation}
by $\nabla y_\de\cdot A_\de\nabla y_\de=(\pt_{x_1}y_\de)^2+x_2^\al (\pt_{x_2}y_\de)^2$ on $Q_\de$. 

Multiplying $y_\de$ on the both sides of \eqref{08.04.7}, integrating on $Q_\de$, integration by parts, we get 
\begin{equation*}
	\begin{split}
		\int_{\Om_\de}(\pt_ty_\de)y_\de\df x\bigg|_{t=0}^{t=T}=\iint_{Q_\de} \left((\pt_ty_\de)^2-\nabla y_\de\cdot A_\de\nabla y_\de\right)\df x\df t. 
	\end{split}
\end{equation*}
Motivated by the previous identities, we introduce the auxiliary functional
\begin{equation*}
	X(t)=\int_{\Om_\de} (\pt_ty_\de)\left(x\cdot \nabla y_\de+\f{2+\al}{4}y_\de\right)\df x\df t, 
\end{equation*}
then 
\begin{equation}\label{08.06.5}
	\begin{split}
		X(t)\bigg|_{t=0}^{t=T} +\f{2-\al}{2}TE(\de;0)\leq \f{1}{2}\iint_{(\pt\Om_\de-\Ga_{\de,2}^\de)\ts (0,T)}(\nabla y_\de\cdot A_\de\nabla y_\de)(x\cdot\nu)\df S\df t. 
	\end{split}
\end{equation}

It is easily verified that
\begin{equation*}
	\begin{split}
		|X(t)|\leq \f{\be}{2}\int_{\Om_\de} (\pt_ty_\de)^2\df x+\f{1}{2\be}\int_{\Om_\de} \left|x\cdot \nabla y_\de+\f{2+\al}{4}y_\de\right|^2\df x
	\end{split}
\end{equation*}
for all $\be>0$. Note that from 
\begin{equation*}
	\begin{split}
		|x\cdot \nabla y_\de|^2=|x_1\pt_{x_1} y_\de +x_2\pt_{x_2}y_\de|^2\leq 2(\pt_{x_1}y_\de)^2+2x_2^\al (\pt_{x_2}y_\de)^2=2\nabla y_\de\cdot A_\de\nabla y_\de 
	\end{split}
\end{equation*}
we get 
\begin{equation*}
	\begin{split}
		&\int_{\Om_\de} \left|x\cdot\nabla y_\de+\f{2+\al}{4}y_\de\right|^2\df x\\
		&\leq 2\int_{\Om_\de} (\nabla y_\de\cdot A_\de\nabla y_\de)\df x+\left(\f{2+\al}{4}\right)^2\int_{\Om_\de} y_\de^2\df x+\f{2+\al}{2}\int_{\Om_\de} (x\cdot \nabla y_\de) y_\de\df x.  
	\end{split}
\end{equation*}
Together with this and 
\begin{equation*}
	\begin{split}
		\int_{\Om_\de} (x\cdot \nabla y_\de)y_\de\df x=-\int_{\Om_\de} y_\de^2\df x, 
	\end{split}
\end{equation*}
we deduce that 
\begin{equation*}
	\begin{split}
		\int_{\Om_\de} \left|x\cdot\nabla y_\de+\f{2+\al}{4}y_\de\right|^2\df x
		&\leq 2\int_{\Om_\de} (\nabla y_\de\cdot A_\de\nabla y_\de)\df x
	\end{split}
\end{equation*}
by $(\f{2+\al}{4})^2-\f{2+\al}{2}=\f{(2+\al)(\al-6)}{16}\leq 0$. Hence
\begin{equation*}
	|X(t)|\leq \f{\be}{2}\int_{\Om_\de} (\pt_ty_\de)^2\df x+\f{1}{\be}\int_{\Om_\de} (\nabla y_\de\cdot A_\de \nabla y_\de)\df x. 
\end{equation*}
Take $\be=\sqrt{2}$, we get 
\begin{equation*}
	|X(t)|\leq \sqrt{2}E(\de;t). 
\end{equation*}
Together with this and \eqref{08.06.5}, we obtain 
\begin{equation}\label{08.06.6}
		\left(\f{2-\al}{2}T-\sqrt{2}\right)E(\de;0)\leq \f{1}{2}\iint_{(\pt\Om_\de-\Ga_{\de,2}^\de)\ts (0,T)}(\nabla y_\de\cdot A_\de\nabla y_\de)(x\cdot \nu)\df S\df t. 
\end{equation}
From \eqref{08.05.17} and $E(\de;0)=E(0)$, since the constants in \eqref{08.06.6} do not depend on $\de\in (0,\de_0)$, passing to the limit in \eqref{08.06.6} by using the convergence result \eqref{08.05.17}, together with $E(\delta;0)=E(0)$,  yields 
\begin{equation*}
	\begin{split}
		\left(\f{2-\al}{2}T-\sqrt{2}\right)E(0)\leq \f{1}{2}\iint_{(\pt\Om-\Ga_{2}^0)\ts (0,T)}(\nabla y\cdot A\nabla y)(x\cdot \nu)\df S\df t, 
	\end{split}
\end{equation*}
where 
\begin{equation*}
	E(t)=\f{1}{2}\iint_Q \left((\pt_ty)^2+\nabla y\cdot A\nabla y\right)\df x, \mbox{ for all } t\in [0,T]. 
\end{equation*}
This is the following theorem. 

\begin{theorem}\label{08.06.T1}
	Let $y^0,y^1\in C_0^\iy(\Om)$. Then for $T>\f{2\sqrt{2}}{2-\al}$, we have the following observability inequality 
	\begin{equation*}
		\left(
		\frac{2-\alpha}{2}T-\sqrt2
		\right)
		E(0)
		\le
		\frac12
		\iint_{(\partial\Omega-\Gamma_2^0)\times(0,T)}
		(\nabla y\cdot A\nabla y)
		(x\cdot\nu)
		\,dSdt.
	\end{equation*}
	where 
	\begin{equation*}
		E(t)=\f{1}{2}\iint_Q \left((\pt_ty)^2+\nabla y\cdot A\nabla y\right)\df x, \mbox{ for all } t\in [0,T], 
	\end{equation*}
	and $y$ is the weak solution of \eqref{01.08.1} with initial data $(y^0,y^1)$. 
\end{theorem}

Finally, let  $ (y^0,y^1)\in H_0^1(\Om;w)\ts L^2(\Om)$. Then there exist sequences
$y_h^0,y_h^1\in C_0^\infty(\Omega)$, 
such that
$\|y_h^0-y^0\|_{H_0^1(\Om;w)}\to0$,
and 
$\|y_h^1-y^1\|_{L^2(\Om)}\to0$ as $h\to\infty$.
Hence, from Corollary \ref{08.07.C2}, we obtain 
\begin{equation}\label{08.07.7}
	\left\|
	\frac{\partial y_h}{\partial \nu_A}
	\right\|_{L^2((\partial\Omega-\Gamma_N^0)\times(0,T))}
	\leq
	CT
	\left(
	\|y_h^0\|_{H_0^1(\Omega;w)}
	+
	\|y_h^1\|_{L^2(\Omega)}
	\right), 
\end{equation} 
where the positive constant $C$ is independent of $h$ and $T$, where $y_h$ is the weak solution of \eqref{01.08.1} with initial data $(y_h^0,y_h^1)$. 
Then there exists a subsequence of $\{y_h\}_{h\in\N^*}$, still denoted by itself, and $g$ in $L^2((\pt\Om-\Ga_2^0)\ts (0,T))$ such that 
\begin{equation}\label{08.07.8}
	\f{\pt y_h}{\pt \nu_A} \ra g\mbox{ weakly in } L^2((\pt\Om-\Ga_2^0)\ts (0,T)) \mbox{ as } h\ra \iy. 
\end{equation}

It is easily verified that 
\begin{equation*}
	\begin{split}
		y_h
		&\ra y \mbox{ weak star in } L^\iy (0,T; H_0^1(\Om;w)), \\
		\pt_ty_h
		&\ra \pt_ty \mbox{ weak star in } L^\iy(0,T; L^2(\Om))
	\end{split}
\end{equation*}
by \eqref{08.07.6}, where $y$ is the weak solution of \eqref{01.08.1} with initial data $(y^0,y^1)$. 
Let $\psi\in C_0^\iy(\Ga_1^{-1}\ts (0,T))$, taking $\Psi\in C^\iy(\ol Q)$ with $\Psi|_{\Ga_1^{-1}\ts (0,T)}=\psi$ and $\Psi=0$ on $\pt Q-(\Ga_1^{-1}\ts (0,T))$, multiplying $\Psi$ on the both sides of \eqref{01.08.1}, integrating on $Q$, integration by parts, we obtain
\begin{equation*}
	\begin{split}
		\iint_{\Ga_1^{-1}\ts (0,T)} \f{\pt y_h}{\pt \nu_A}\psi\df x\df t=\iint_Q (\pt_ty_h)\pt_t\Psi\df x\df t-\iint_Q A\nabla y_h\cdot \nabla \Psi\df x\df t, 
	\end{split}
\end{equation*}
and 
\begin{equation*}
	\iint_{\Ga_1^{-1}\ts (0,T)} \f{\pt y}{\pt \nu_A}\psi\df x\df t=\iint_Q (\pt_ty)\pt_t\Psi\df x\df t-\iint_Q A\nabla y\cdot \nabla \Psi\df x\df t. 
\end{equation*}
Passing $h\ra \iy$, we get 
\begin{equation*}
	\f{\pt y}{\pt \nu_A}=g \mbox{ in } L^2(\Ga_1^{-1}\ts (0,T)). 
\end{equation*}
By the same argument as above, we get 
\begin{equation*}
	\f{\pt y}{\pt \nu_A}=g \mbox{ in  } L^2((\Ga_1^1\cup \Ga_2^1)\ts (0,T)). 
\end{equation*}
For any $h,n\in\N^*$,  $y_h-y_n$ is the solution of \eqref{01.08.1} with initial date $(y_h^0-y_n^0, y_h^1-y_n^1)$, from Corollary \ref{08.07.C2}, we obtain 
\begin{equation*}
	\left\|
	\frac{\partial (y_h-y_n)}{\partial \nu_A}
	\right\|_{L^2((\partial\Omega-\Gamma_N^0)\times(0,T))}
	\leq
	CT
	\left(
	\|y_h^0-y_n^0\|_{H_0^1(\Omega;w)}
	+
	\|y_h^1-y_n^1\|_{L^2(\Omega)}
	\right). 
\end{equation*} 
Hence 
\begin{equation*}
	\f{\pt y_h}{\pt \nu_A} \ra \f{\pt y}{\pt \nu_A} \mbox{ strongly in }L^2((\pt \Om-\Ga_2^0)\ts (0,T)). 
\end{equation*}
Which together with $y_h^0\ra y^0$ and $y_h^1\ra y^1$ (i.e., $\f{1}{2}(\|y_h^0\|_{H_0^1(\Om;w)}^2+\|y_h^1\|_{L^2(\Om)}^2)\ra \f{1}{2}(\|y^0\|_{H_0^1(\Om;w)}^2+\|y^1\|_{L^2(\Om)}^2)$  shows the following theorem. 

\begin{theorem}\label{08.06.T2}
	Let $y^0\in H_0^1(\Om;w)$ and $y^1\in L^2(\Om)$. If
	$T>\frac{2\sqrt2}{2-\alpha}$,
	then the following observability inequality holds:
	\begin{equation*}
	\left(\frac{2-\alpha}{2}T-\sqrt2
	\right)
	E(0)
	\le
	\frac12
	\iint_{(\partial\Omega-\Gamma_2^0)\times(0,T)}
	(\nabla y\cdot A\nabla y)
	(x\cdot\nu)
	\df S\df t,
	\end{equation*} 
	where
	\begin{equation*} 
	E(t)
	=
	\frac12
	\int_\Omega
	\left(
	(\partial_t y)^2
	+
	\nabla y\cdot A\nabla y
	\right)
	\df x,
	\qquad
	t\in[0,T],
	\end{equation*} 
	and $y$ is the weak solution of \eqref{01.08.1} corresponding to the initial data $(y^0,y^1)$.
\end{theorem}


\begin{thebibliography}{}  
	
	
	
	
	
	\bibitem{Alabau}
	F. Alabau-Boussouira, P. Cannarsa and G. Fragnelli, Carleman estimates for degenerate parabolic operators with applications to null controllability, {\it J. Evol. Equ.}, 6(2006), 161-204.
	
	\bibitem{Alabau1}
	F. Alabau-Boussouira, P. Cannarsa and G. Leugering, Control and stabilization of degenerate wave equations, {\it SIAM J. Control Optim.}, 55 (2017),  2052–2087. 
	
	\bibitem{Akil}
	M. Akil, G. Fragnelli and  S. Ismail, Boundary observability and null-controllability for non-autonomous degenerate hyperbolic equations via energy estimates, {\it J. Evol. Equ.}, 25(2025),  https://doi.org/10.1007/s00028-025-01122-5.  
	
	
	\bibitem{Bai}
	J. Bai and S. Chai, Exact controllability of wave equations with interior degeneracy and
	one-sided boundary control, {\it J. Syst. Sci. Complex.}, 36(2023), 656-671.
	
	\bibitem{Buffe}
	R. Buffe, K.D. Phung and A. Slimani, An optimal spectral inequality for degenerate operators, {\it SIAM J. Control Optim.}, 62(2024), 2506-2528. 
	
	
	
	
	
	\bibitem{Buttazzo}
	G. Buttazzo and P. Guasoni, Shape optimization problems over classes of convex domains,
	{\it J. Convex Anal.}, 4(1997),  343-351.
	
	\bibitem{Cannarsa}
	P. Cannarsa, P. Martinez and C. Urbani, Bilinear Control of a Degenerate Hyperbolic Equation, {\it SIAM J. Control Optim.}, 55(2023), https://doi.org/10.1137/22M148745. 
	
	
	\bibitem{Cannarsa1}
	P. Cannarsa, P. Martinez and J. Vancostenoble, The cost of controlling weakly degenerate parabolic equations by boundary control, {\it Math. Control Relat. Fields}, 7(2017), 171-211. 
	
	 
	
	\bibitem{Chenais}
	D. Chenais, On the existence of a solution in a domain identification problem, J. Math. Anal.
	Appl., 52 (1975), pp. 189–219.
	
	    
	\bibitem{Dunford}
	N. Dunford and J.T. Schwartz, {\it Linear operators, Part II: Spectral Theory}, Interscience Publishers, New York, London, 1963. 
	
	  \bibitem{Evans}
	 L.C. Evans, {\it Partial Differential Equations}, American Mathematical Society, New York, 2010. 
	
	
	\bibitem{Fabes}
	E.B. Fabes, C.E. Kenig, and R.P. Serapioni, The local regularity of solutions of degenerate elliptic equations, {\it Comm. Partial Differential Equations}, 7(1982), 77-116.
	
	\bibitem{Fragnelli}
	G. Fragnelli, D. Mugnai and  A. Sbai, Boundary Controllability for Degenerate/Singular Hyperbolic Equations in Nondivergence Form with Drift, {\it Appl. Math. Optim.}, 91(2025), https://doi.org/10.1007/s00245-025-10236-8.
	 
	\bibitem{GC}
	J. Garcia-Cuerva and J.R. de Francia, {\it Weighted Norm Inequalities and Related Topics}, North-Holland Publishing Co., Amsterdam, 1985.
	
	\bibitem{Greco}
	L. Greco, An approximation theorem for the $\Ga^-$-convergence of degenerate quadratic functionals, {\it
		Riv. Mat. Pura Appl.}, 7(1990), 53-80.
		
	\bibitem{Gueye}
	M. Gueye, Exact boundary controllability of $1$-d parabolic and hyperbolic degenerate equations, {\it SIAM J. Control Optim.}, 52(2014), 2037-2054. 
	
	\bibitem{Guo2}
	B.Z. Guo and D. Yang, Some compact classes of open sets under Hausdorﬀ distance and
	application to shape optimization, {\it SIAM J. Control Optim.}, 50(2012), 222-242.
	
	\bibitem{Guo1}
	B.Z. Guo and D.H. Yang, On convergence of boundary Hausdorff measure and application to a
	boundary shape optimization problem, {\it SIAM J. Control Optim.}, 51(2013), 253-272.
	
	\bibitem{Guo}
	B.-Z. Guo, D.-H. Yang and J. Zhong, A Shape Design Approximation for Degenerate
	Partial Differential Equations with Degenerate Part Boundary and Application, {\it J. Optim. Theory Appl.}, 209(2026),  https://doi.org/10.1007/s10957-026-03011-4.
	
	
	
	\bibitem{He}
	Y. He and B.Z. Guo, The existence of optimal solution for a shape optimization problem on
	starlike domain, {\it J. Optim. Theory Appl.}, 152 (2012), 21-30.
	
	
	\bibitem{Heinonen}
	J. Heinonen and T. Kilpenl\"ainen and O. Martio, {\it Nonlinear Potential Theory of Degenerate Elliptic Equations}, Oxford University Press, New York, 1993.
	
	\bibitem{Henrot}
	A. Henrot, {\it Extremum Problems for Eigenvalues of Elliptic Operators}, Birkh\"auser, Basel, 2006.
	
	\bibitem{Komornik}
	V. Komornik, {\it Exact Controllability and Stabilization:  The Multiplier Method}, Wiley-Masson, Paris, 1994.
	
	\bibitem{Lasiecka}
	I. Lasiecka and R. Triggiani, Riccati equations for hyperbolic partial differential equations with $L^2(0,T; L^2(\Ga))$ Dirichlet boundary terms, {\it SIAM J. Control Optim.}, 24(1986),
	884–925.
	
	\bibitem{Lasiecka1}
	I. Lasiecka, R. Triggiani, and P.-F. Yao, An observability estimate in $L^2(\Om)\ts H^{-1}(\Om)$ 
	for second-order hyperbolic equations with variable coefficients, in {\it Control of Distributed
	Parameter and Stochastic Systems}, Kluwer Academic Publishers, Boston, MA, 1999, 71–
	78.
	
	\bibitem{Lions}
	J.L. Lions, Exact controllability, stabilizability and perturbations for distributed systems, {\it SIAM Rev.}, 30(1988), 1-68.
	
	\bibitem{Lu}
	Q. L\"u and X. Zhang, {\it Exact Controllability for a Refined Stochastic Wave Equation}, In: Mathematical Control Theory for Stochastic Partial Differential Equations. Probability Theory and Stochastic Modelling, vol 101. Springer, Cham.
	
	\bibitem{Privat}
	Y. Privat, E. Tr\'elat, and E. Zuazua, Optimal observability of the multi-dimensional wave and Schr\"odinger equations in quantum ergodic domains, {\it J. Eur. Math. Soc.}, 18(2016), 1043-1111.
	 
	\bibitem{Tiba}
	C.M. Murea and D. Tiba, Topological optimization via cost penalization,  {\it Topol. Method Nonl. An.}, 54(2019), 1023-1050.
	  
	 \bibitem{Weiss}
	 M. Tucsnak and G. Weiss, {\it Observation and Control for Operator Semigroups}, Birkh\"auser, Basel, 2009.
	
	
	\bibitem{Wang}
	G. Wang, L. Wang, and D. Yang, Shape optimization of an elliptic equation in an exterior
	domain, {\it SIAM J. Control Optim.}, 45 (2006), pp. 532–547.
	
	 
	
	\bibitem{Yang}
	D.-H. Yang, B.-Z. Guo, Null controllability of a class of degenerate parabolic equations subject to Dirichlet boundary conditions, {\it Syst. Control Lett.}, 209(2026), 106361.
	
	\bibitem{Yang6}
	D.H. Yang and B.-Z. Guo, Approximation of Degenerate Hyperbolic Equations
	with Interior Degeneracy and Applications to
	Controllability,  https://arxiv.org/pdf/2605.09069.
	
	\bibitem{Yang4}
	D.H. Yang, W. Wu, B.Z. Guo, and S. Chai, On exact controllability for a class of 2-D
	Grushin hyperbolic equations, {\it J. Differ. Equ.}, 448(2025), 113710.
	
	\bibitem{Yang5}
	D. Yang and J. Zhong, Observability inequality of backward stochastic heat equations
	for measurable sets and its applications, {\it SIAM J. Control Optim.}, 54 (2016), 1157–1175.
	
	\bibitem{Yang1}
	D.H. Yang  and J. Zhong, Non-homogeneous boundary value problems for second-order degenerate hyperbolic equations and their applications, https://arxiv.org/abs/2602.07271.
	
	 \bibitem{Yang2}
	 D.H. Yang and J. Zhong, Hidden Boundary Trace Regularity and an
	 Observability Estimate with Interior Remainder for
	 Boundary-Degenerate Hyperbolic Equations,   https://arxiv.org/pdf/2605.01254. 
	 
	 \bibitem{Yang3}
	 D.H. Yang and J. Zhong, Shape-Design Approximation for a Class of Degenerate
	 Hyperbolic Equations with a Degenerate Boundary
	 Point and Its Application to Observability,  https://arxiv.org/pdf/2603.10428. 
	
	\bibitem{Yao}
	P.-F. Yao, On the observability inequalities for exact controllability of wave equations
	with variable coefficients, {\it SIAM J. Control Optim.}, 37(1999), 1568–1599.
	
	\bibitem{Zhang}
	M. Zhang and H. Gao, Persistent regional null controllability of some degenerate wave
	equations, {\it Math. Methods Appl. Sci.}, 40 (2017), 5821–5830.
	
	\bibitem{Zuazua}
	E. Zuazua, Exact Controllability and Stabilization of the Wave Equation, Springer,
	Cham, 2024.
	 
	
	 
\end{thebibliography}
\end{document}